\documentclass[11pt]{ip-journal}

\usepackage{amsmath}
\usepackage{amsfonts}
\usepackage{amssymb}
\usepackage{amsthm}
\usepackage{esint}
\usepackage{mathrsfs}
\usepackage{mathtools}
\usepackage{tensor}

\usepackage{ucs}
\usepackage[utf8x]{inputenc}
\newcommand{\RR}{\mathbf{R}}

\renewcommand{\SS}{\mathbf{S}}

\newcommand{\JJ}{\mathbf{J}}

\newcommand{\cA}{\mathcal{A}}

\newcommand{\cE}{\mathcal{E}}

\newcommand{\cH}{\mathcal{H}}
\newcommand{\cI}{\mathcal{I}}
\newcommand{\cJ}{\mathcal{J}}
\newcommand{\cL}{\mathcal{L}}
\newcommand{\cM}{\mathcal{M}}
\newcommand{\cN}{\mathcal{N}}

\newcommand{\cP}{\mathcal{P}}
\newcommand{\cQ}{\mathcal{Q}}
\newcommand{\cR}{\mathcal{R}}
\newcommand{\cS}{\mathcal{S}}

\DeclareMathOperator{\ind}{ind}
\DeclareMathOperator{\nul}{nul}

\DeclareMathOperator{\grassmanian}{Gr}

\DeclareMathOperator{\support}{spt}

\DeclareMathOperator{\dist}{dist}

\DeclareMathOperator{\intr}{int}

\DeclareMathOperator{\loc}{loc}

\DeclareMathOperator{\divg}{div}

\DeclareMathOperator{\sing}{sing}
\DeclareMathOperator{\reg}{reg}

\DeclareMathOperator{\ricc}{Ric}
\DeclareMathOperator{\inj}{inj}
\DeclareMathOperator{\sect}{sect}
\DeclareMathOperator{\sff}{\mathrm{I\!I}}

\newcommand{\closure}[1]{\overline{#1}}

\newcommand{\indt}[1]{{\bf 1}_{#1}}

\newcommand{\restr}{\mathbin{\vrule height 1.6ex depth 0pt width
0.13ex\vrule height 0.13ex depth 0pt width 1.3ex}}

\newcommand{\weaklyto}{\rightharpoonup}
\newcommand{\duline}[1]{\underline{\underline{#1}}}
\newcommand{\energyunit}{h_0}

\theoremstyle{plain} \newtheorem{defi}{Definition}
\theoremstyle{plain} \newtheorem{rema}[defi]{Remark}
\theoremstyle{plain} \newtheorem{theo}[defi]{Theorem}
\theoremstyle{plain} \newtheorem{prop}[defi]{Proposition}
\theoremstyle{plain} \newtheorem{coro}[defi]{Corollary}
\theoremstyle{plain} \newtheorem{lemm}[defi]{Lemma}
\theoremstyle{plain} \newtheorem{exam}[defi]{Example}
\theoremstyle{plain} \newtheorem{conj}[defi]{Conjecture}
\theoremstyle{plain} \newtheorem*{theo*}{Theorem}
\theoremstyle{plain} 
\theoremstyle{plain} 
\theoremstyle{definition} \newtheorem*{clai}{Claim}

\numberwithin{defi}{section} % number theorems/etc according to the current section number
\numberwithin{equation}{section} % number equations according to the current section number

\allowdisplaybreaks

\title[Allen-Cahn min-max on surfaces]{Allen-Cahn min-max on surfaces}
\author{Christos Mantoulidis}
\address{Department of Mathematics, Massachusetts Institute of Technology, Cambridge, MA 02139}
\email{c.mantoulidis@mit.edu}

\begin{document}

\maketitle

\begin{abstract}
	We use a min-max procedure on the Allen-Cahn energy functional to construct geodesics on closed, 2-dimensional Riemannian manifolds, as motivated by the work of Guaraco \cite{Guaraco15}. Borrowing classical blowup and curvature estimates from geometric analysis, as well as novel Allen-Cahn curvature estimates due to Wang-Wei \cite{WangWei17}, we manage to study the fine structure of potential singular points at the diffuse level, and show that the problem reduces to that of understanding ``entire'' singularity models constructed by del Pino-Kowalczyk-Pacard \cite{DelPinoKowalczykPacard13} with Morse index 1. The argument is completed by a Morse index estimate on these singularity models. 
\end{abstract}

\tableofcontents

\section{Introduction} \label{sec:introduction}

The Allen-Cahn equation is an elliptic partial differential equation describing phase separation in multi-component alloy systems. It is:
\begin{equation} \label{eq:pde}
	\varepsilon \Delta u = \varepsilon^{-1} W'(u),	
\end{equation}
where $W$ is a double-well energy potential, and $\varepsilon > 0$.

\begin{defi} \label{defi:double.well.potential}
	A smooth map $W : \RR \to \RR$ is called a double-well potential provided:
	\begin{enumerate}
		\item $W$ is nonnegative, and vanishes at $t = \pm 1$:
			\begin{equation} \label{assu:double.well.potential.nonnegative} \tag{H1}
				W \geq 0, \; W(t) = 0 \iff t = \pm 1 \text{;}	
			\end{equation}
		\item $W$ has a unique critical point between its global minima, at $t = 0$, which is nondegenerate:
			\begin{equation} \label{assu:double.well.potential.saddle} \tag{H2}
				tW'(t) < 0 \text{ for } 0 < |t| < 1	\text{, and } W''(0) \neq 0 \text{;}
			\end{equation}
		\item $W$ is strictly convex near $\pm 1$:
			\begin{equation} \label{assu:double.well.potential.convex} \tag{H3}
				W''(t) \geq \kappa > 0 \text{ for } |t| > 1-\alpha, \; \alpha \in (0,1) \text{;}
			\end{equation}
	\end{enumerate}
	in this paper we will additionally assume:
	\begin{enumerate}
		\item[4)] $W$ is even:
			\begin{equation} \label{assu:double.well.potential.even} \tag{H4}
				W(t) = W(-t) \text{ for } t \in \RR.
			\end{equation}
	\end{enumerate}
\end{defi}

\begin{exam} \label{exam:canonical.potential}
	The standard double-well potential is
	\[ W(t) = \frac{1}{4} (1-t^2)^2; \]
	the corresponding equation \eqref{eq:pde} is
	\[ \varepsilon \Delta u = \varepsilon^{-1} (u^3 - u). \]
\end{exam}

There are strong parallels between the study of minimal surfaces and the study of phase transitions. Roughly speaking, the level sets
\[ \{ u_\varepsilon = t \}, \; t \in (-1, 1), \]
of solutions $u_\varepsilon$ to \eqref{eq:pde} converge (in a sense to be made precise) to minimal hypersurfaces, as $\varepsilon \to 0$. There has been a lot of work done in exploration of the analogy between these two equations; we refer the reader, as a starting point, to study the works of del Pino-Kowalczyk-Wei \cite{DelPinoKowalczykWei13}, del Pino-Kowalczyk-Wei-Yang \cite{DelPinoKowalczykWeiYang10}, Hutchinson-Tonegawa \cite{HutchinsonTonegawa00}, Tonegawa-Wickramasekera \cite{TonegawaWickramasekera12}, and references therein for more information.

In recent novel work, M. Guaraco \cite{Guaraco15} successfully used min-max theory together with the theory of phase transitions to give an alternate proof of the existence of minimal hypersurfaces in closed Riemannian manifolds of ambient dimension $3, 4, 5, 6, 7$. This result was originally proved in this generality by R. Schoen and L. Simon in \cite{SchoenSimon81} using the Almgren-Pitts min-max theory. 

Guaraco's method cannot extend to the construction of geodesics on two-dimensional surfaces, due to the potential formation of singular geodesic junctions. H. del Rio Guerra, C. Garza-Hume, and P. Padilla had successfully used the method of phase transitions to construct embedded geodesics in \cite{DelRioGuerraGarzaHumePadilla03} subject to a nonpositivity or nonnegativity condition on the Gauss curvature of the two-dimensional surface, which precisely allows to rule out singularities.

In this work, we circumvent the complication introduced by the formation of singularities by employing a diligent study of the diffuse (small $\varepsilon > 0$) problem prior to the singularity formation, instead of the limiting problem ($\varepsilon \to 0$) where singularities have formed. This is closer in spirit to the work of Y. Tonegawa and N. Wickramasekera \cite{TonegawaWickramasekera12} than to \cite{Guaraco15}. We are indeed successful in lifting the curvature condition---at the expense of embeddedness---by understanding the precise nature of the singularities that may occur:

\begin{theo} \label{theo:minmax.construction}
    Any closed Riemannian manifold $(\Sigma^2, g)$ admits a closed immersed geodesic with at most one self-intersection; the intersection, if it exists at all, is transverse.
\end{theo}

\begin{rema}
	Such $(\Sigma^2, g)$ all contain closed and embedded geodesics. (See, e.g., the work of L. Lyusternik and L. Schnirelmann \cite{LyusternikSchnirelmann47}.) The point is that we are able to make use of phase transitions in this context.
\end{rema}

The most common way to study singularity formation in geometric analysis is to ``blow up'' the picture near the singular point and obtain, in the limit, a so-called {\bf singularity model}. This model is then studied in isolation, and results about it are used to understand the structure of the original singularity.

In the phase transition model, singularity models will correspond to nonconstant solutions $u : \RR^n \to \RR$ of \eqref{eq:pde}, with $|u| < 1$ and $\varepsilon = 1$. When $n = 2$, finite-index singularity models turn out to coincide with those constructed by M. del Pino, M. Kowalczyk, and F. Pacard in \cite{DelPinoKowalczykPacard13}. In Section \ref{sec:m2k.index.lower.bound} we confirm that the complexity of these singularity models at infinity has a direct effect on their Morse index (see Definition \ref{def:morse.index}). More specifically, we prove the following theorem on the Morse index of elements in the moduli space $\cM_{2k}$ of $2k$-ended solutions of \eqref{eq:pde} in $\RR^2$ with $\varepsilon = 1$ (see Definition \ref{defi:m2k}).

\begin{theo} \label{theo:m2k.index.lower.bound}
	Let $u \in \cM_{2k}$, $k \geq 2$. Then $\ind(u) \geq k-1$.
\end{theo}

(Note added in proof process of manuscript: Liu-Wei \cite{LiuWei18} recently claimed that, for the sine-Gordon equation, $\ind(u) = \tfrac12 k(k-1)$.) 

Such results are common in minimal surface theory: the index of a minimal surface is related to its topology; we refer the reader to the work of O. Chodosh and D. Maximo \cite{ChodoshMaximo16} for more information on this. Theorem \ref{theo:m2k.index.lower.bound}, thus, helps further solidify the analogies between phase transitions and minimal surfaces. Closest in spirit to this theorem in minimal surface theory is a result of A. Grigor\'yan, Y. Netrusov, and S.-T. Yau \cite[Section 6.4]{GrigoryanNetrusovYau04}.

Given Theorem \ref{theo:m2k.index.lower.bound}, the strategy for the proof of Theorem \ref{theo:minmax.construction} is as follows (we refer the reader to Section \ref{sec:background} for all the notation):
\begin{enumerate}
	\item Produce a sequence $\{(u_i, \varepsilon_i)\}_{i=1,2,\ldots} \subset C^\infty(\Sigma) \times (0, \infty)$, $\lim_i \varepsilon_i = 0$, with controlled $E_{\varepsilon_i}$-energy, and where each $u_i$ is a nontrivial critical point of $E_{\varepsilon_i}$ with Morse index $\leq 1$.
	\item Up to a subsequence, show $\lim_i V_{\varepsilon_i}[u_i] = V^\infty$ with $\support \Vert V^\infty \Vert$ a geodesic network with at most one singular junction, $p_*$.
	\item Study the convergence $V_{\varepsilon_i}[u_i] \weaklyto V^\infty$ near $p_*$ using blowups reminiscent of minimal surface theory, and Theorem \ref{theo:m2k.index.lower.bound}'s classification of Morse index-1 singularity models in $\RR^2$, to show all singular points have density $2$ and are thus intersection points of immersed geodesics (see Appendix \ref{sec:app.gmt}).
\end{enumerate}

\begin{rema}
	After the completion of this manuscript, O. Chodosh and the author were able to obtain the following partial generalization of the aforementioned three-step program: one may study critical points $\{(u_i, \varepsilon_i)\}_{i=1,2,\ldots} \subset C^\infty(\Sigma) \times (0, \infty)$, $\lim_i \varepsilon_i = 0$, with Morse index $\leq I_0$, and conclude that:
	\begin{enumerate}
		\item[i)] at most $I_0$ singular points form, and
		\item[ii)] the sum, over each singularity, of its density minus 1, is $\leq I_0$.
	\end{enumerate}
	This result will not be presented in this paper for the sake of brevity (it doesn't require the introduction of new ideas, just an inductive argument) and because it has more limited geometric appeal (such limiting geodesic networks aren't necessarily smoothly immersed).
\end{rema}

The first step of this program appears verbatim in the work of M. Guaraco \cite{Guaraco15}. The second step follows from a classical covering argument in minimal surface theory. The difficulty is in the third step, where we need to perform blowups carefully: we cannot afford to lose density information in going from the local picture of $(u_i, \varepsilon_i)$ near $p_*$ to the blown up entire solution over $\RR^2$ with $\varepsilon = 1$. This is a particularly delicate matter, seeing as to how interfaces don't generally appear in the same $O(\varepsilon_i)$-scale. This is the content of Proposition \ref{prop:index.1.main}, which effectively shows that \emph{the convergent interfaces of $\{ u_i \}_{i=1,2,\ldots}$ all appear in the $O(\varepsilon_i)$-scale around the singularity} by way of a \emph{uniqueness of tangent cones}-type result. In order to prove this result, we make use of recent novel estimates of K. Wang and J. Wei \cite{WangWei17} for stable solutions of \eqref{eq:pde}.

\thanks{{\bf Acknowledgements.} The author would like to thank O. Chodosh for bringing the problem to his attention, and R. Schoen, R. Mazzeo, K. Wang, L. Simon, M. Guaraco, D. Cheng, and C. Li for useful conversations. The author would also like to thank the Department of Mathematics at the University of California, Irvine, where part of this research was carried out, as well as Stanford University and the Ric Weiland Graduate Fellowship at Stanford, which partially supported this research, alongside NSF grant DMS-1613603.

\section{Background} \label{sec:background}

\subsection{Notation} 

We point out the following notational choices, which may deviate from some papers in the literature:

\begin{enumerate}
	\item \emph{Completeness} and \emph{compactness} refer to metric completeness. Thus, a complete or compact manifold may or may not have boundary.
	\item Compact manifolds without boundary are referred to as being \emph{closed}.
	\item $C^k$ spaces are all endowed with the Banach space norm
		\[ |f|_{C^k(\Omega)} = \sum_{j=0}^k \Vert f \Vert_{L^\infty(\Omega)}; \]
		$C^k_{\loc}(\Omega)$ denotes the set of all functions $f : \Omega \to \RR$ which are in $C^k(K)$ for every compact $K \subset \subset \Omega$, and $C^k_c(\Omega)$ denotes the set of all functions $f : \Omega \to \RR$ which are in $C^k(\Omega)$ and $\support f \subset \subset \Omega$.
	\item $C^\infty(\Omega)$ denotes the Fr\'echet space generated by the intersection of all $C^k(\Omega)$, $k \in \{ 0, 1, \ldots \}$; $C^\infty_{\loc}(\Omega)$, $C^\infty_c(\Omega)$ are defined analogously.
	\item See Appendix \ref{sec:app.gmt} for notation relating to geometric measure theory.
\end{enumerate}

\subsection{Variational structure} 

\eqref{eq:pde} arises as the Euler-Lagrange equation for critical points of the energy functional
\begin{equation} \label{eq:energy}
E_\varepsilon[u] = \int \left( \frac{\varepsilon}{2} |\nabla u|^2 + \varepsilon^{-1} W(u) \right) \, d\upsilon \text{;} \end{equation} 
i.e., zeroes $u \in C^\infty_{\loc}$ of the first variation functional,
\begin{equation} \label{eq:first.variation}
	\delta E_\varepsilon[u]\{\zeta\} = \int \left( \varepsilon \langle \nabla u, \nabla \zeta \rangle - \varepsilon^{-1} W'(u) \zeta \right) \, d\upsilon.
\end{equation}
One can make precise the notions of {\bf stability} and {\bf Morse index} in the Allen-Cahn setting by turning to the {\bf second variation} operator,
\begin{equation} \label{eq:second.variation}
	\delta^2 E[u]\{\zeta, \psi\} \triangleq \int_M \left( \varepsilon \langle \nabla \zeta, \nabla \psi \rangle +  \varepsilon^{-1} W''(u) \zeta \psi \right) \, d\upsilon.
\end{equation}
Associated to $\delta^2 E[u]$ is the linear elliptic operator $- \varepsilon \Delta + \varepsilon^{-1} W''(u)$, corresponding to the {\bf Jacobi operator} in minimal surface theory. The Morse index of a critical point measures the number of linearly independent unstable directions for energy. From a physical perspective, unstable critical points are a lot less likely to be observed than stable ones. These notions are all standard in the compact setting, but one needs to be more careful in the noncompact setting. 

\begin{defi}[Stability, Morse index] \label{def:stable} \label{def:morse.index}
   Suppose $(\Sigma^n, g)$ is a complete Riemannian manifold, and $U \subset \Sigma \setminus \partial \Sigma$ is open. A critical point $u$ of $E_\varepsilon \restr U$ is said to be (linearly) stable on an open subset $U' \subset U$ if
    \[ \delta^2 E_\varepsilon[u]\{\zeta, \zeta\} \geq 0 \text{ for every } \zeta \in C^\infty_c(U'). \]
    It is said to have Morse index $k$ on $U'$, denoted $\ind(u; U') = k$, provided
	\begin{multline*}
		\max \{ \dim V : V \subset C^\infty_c(U') \text{ is a vector space such that} \\
		\delta^2 E_\varepsilon[u]\{ \zeta, \zeta \} < 0 \text{ for all } \zeta \in V \setminus \{0\} \} = k.
	\end{multline*}
	If $u$ is stable on $U'$, then $\ind(u; U') = 0$. We will write $\ind(u)$ when the choice of $U'$ is clear from the context. 
\end{defi}

In Section \ref{sec:app.pde.morse.index.spaces.quadratic.area.growth} of the appendix we show that $C^\infty_c$ above can be successfully replaced by $W^{1,2}_0$ under the assumption of quadratic area growth, and give access to many similar theorems as in the case of compact domains. Note that $\ind(u; U')$, as defined here, coincides with $\ind(\delta^2 E_\varepsilon[u] \restr U; U')$ in the appendix for $\varepsilon = 1$. Implicit in the notation $\ind(u; U') = k$ is a choice for the parameter $\varepsilon$; this choice will always refer to the unique $\varepsilon$ for which $u$ is a critical point of $E_\varepsilon \restr U$.

\subsection{Geometric structure} 

To any $u \in C^{\infty}_{\loc}(\Sigma)$ we associate the following ($n-1$)-{\bf varifold}, i.e., a Radon measure on the Grassmanian of $(n-1)$-planes of $\Sigma$:
\[ V_\varepsilon[u]\{f\} \triangleq \int_\Sigma \varepsilon |\nabla u(x)|^2 f(x, T_x) \, d\upsilon(x), \; f \in C^0_c(\grassmanian_{n-1}(\Sigma)); \]
here $T_x$ denotes the tangent hyperplane at $x$ to the level set $\{ u = u(x) \}$. See Section \ref{sec:app.gmt} in the appendix for a brief introduction to the language of varifolds, and \cite{Simon83} for a more thorough treatment.

\begin{rema}
	The integrand is only relevant at $x \in \Sigma$ with $\nabla u(x) \neq 0$, at which points $T_x \{ u = u(x) \}$ is well-defined by the implicit function theorem.	
\end{rema}

We also define a notion of an {\bf enhanced second fundamental form} for the Allen-Cahn problem, a non-symmetric 2-tensor that makes sense for all solutions of the Allen-Cahn equation.

\begin{defi}[Enhanced second fundamental form] \label{def:a.tensor}
    Let $(\Sigma^n, g)$ be a Riemannian manifold, and $u \in C^\infty_{\loc}(\Sigma)$. Let $\Sigma_u = \Sigma \setminus \{ \nabla u = 0 \}$. We define a non-symmetric 2-tensor $\cA$ via
    \[ \cA = |\nabla u|^{-1}(\nabla^2 u - \nabla^2 u(\cdot, \nu) \otimes \nu^\flat), \]
    where $\nu = |\nabla u|^{-1} \nabla u$ is the (oriented) unit normal to the level sets of $u$ and $\nu^\flat$ its dual 1-form, taken with respect to the metric on $\Sigma$.
\end{defi}

The lemma below is the result of a straightforward computation:

\begin{lemm} \label{lemm:a.tensor}
    Let $(\Sigma^n, g)$ be a Riemannian manifold, $u \in C^\infty_{\loc}(\Sigma)$, $x \in \Sigma \setminus \{ \nabla u = 0 \}$, $\mathbf{X}$, $\mathbf{Y} \in T_x \Sigma$. Then
    \[ \cA(\mathbf{X}, \mathbf{Y}) = \frac{\nabla^2 u(\mathbf{X}, \mathbf{Y}^T)}{|\nabla u|}, \]
    where $\mathbf{Y}^T$ denotes the tangential projection of $\mathbf{Y}$ onto the tangent space of the level set of $u$ through $x$. Moreover, the squared norm of the non-symmetric 2-tensor satisfies
    \[ |\cA|^2 = |\sff|^2 + |\nabla^T \log |\nabla u||^2, \]
    where $\sff$ denotes the second fundamental form of the level set of $u$ through $x$, and $\nabla^T$ denotes the tangential gradient on the level set.
\end{lemm}

J. Hutchinson and Y. Tonegawa made precise in \cite{HutchinsonTonegawa00} the sense in which solutions of \eqref{eq:pde} behave like minimal hypersurfaces.

\begin{theo}[Hutchinson-Tonegawa {\cite[Theorem 1]{HutchinsonTonegawa00}}, cf. Guaraco {\cite[Appendix B]{Guaraco15}}] \label{theo:main.known}
    Suppose $(\Sigma^n, g_\infty)$ is a complete Riemannian manifold, $U \subset \Sigma \setminus \partial \Sigma$ is open, $\{ (u_i, \varepsilon_i) \}_{i=1,2,\ldots} \subset C_{\loc}^\infty(U) \times (0,\infty)$, $\lim_i \varepsilon_i = 0$, $\{ g_i \}_{i=1,2,\ldots} \subset \operatorname{Met}(U)$, and $\lim_i g_i = g_\infty$ in $C^\infty_{\loc}(U)$. Assume, additionally, that each $u_i$ is a critical point of $E_{\varepsilon_i} \restr (U, g_i)$ and that 
    \[ |u_i| < 1 \text{ on } U, \text{ and } (E_{\varepsilon_i} \restr (U,g_i))[u_i] \leq E_0, \text{ for } i = 1, 2, \ldots \]
    The following all hold true after perhaps passing to a subsequence:
    \begin{enumerate}
        \item $\lim_i u_i = u_\infty$ in $L^1_{\loc}(U)$, $u_\infty \in BV_{\loc}(U)$, $u_\infty = \pm 1$ a.e. on $U$,
        \item $\lim_i V_{\varepsilon_i}[u_i] \restr \grassmanian_{n-1}(U) = V^\infty \restr \grassmanian_{n-1}(U)$,
        \item $\lim_i (E_{\varepsilon_i} \restr (U, g_i))[u_i] = \Vert V^\infty \Vert(U)$,
        \item for all $t \in (-1,1)$,
            \[ \lim_i \{ u_i = t \} \cap U = \support \Vert V^\infty \Vert \cap U \]
            locally in the Hausdorff topology,
        \item $\lim_i \xi_{\varepsilon_i}[u_i] = 0$ in $L^1_{\loc}(U)$, where
            \[ \xi_{\varepsilon_i}[u_i] \triangleq \frac{\varepsilon_i}{2} |\nabla u_i|^2 - \frac{W(u_i)}{\varepsilon_i}, \]
        \item $\energyunit^{-1} V^\infty \restr \grassmanian_1(U)$ is a stationary integral 1-varifold; here
            \[ \energyunit \triangleq \frac{1}{\sqrt{2}} \int_{-1}^1 \sqrt{W}, \]
        \item $\support \Vert V^\infty \Vert \cap U$ consists of two portions: $\partial^* \{ u_\infty = +1 \} \cap U$, where the multiplicity of $\energyunit^{-1} V^\infty$ is odd; (b) $\support \Vert V^\infty \Vert \cap U \setminus \partial^* \{u_\infty = +1\}$, where the multiplicity of $\energyunit^{-1} V^\infty$ is even.
    \end{enumerate}
\end{theo}

\begin{rema}
	$V_{\varepsilon_i}[u_i]$, and $\xi_{\varepsilon_i}[u_i]$ are computed with respect to $g_i$.
\end{rema}

If the $u_i$ in Theorem \ref{theo:main.known} are endowed with additional variational properties, then the limiting varifold $V^\infty$ may have additional regularity. For instance, if the $u_i$ are additionally assumed to be stable critical points on $U$ then from the combined work of Y. Tonegawa \cite{Tonegawa05}, Y. Tonegawa and N. Wickramasekera \cite{TonegawaWickramasekera12}, and M. Guaraco \cite{Guaraco15}, we know that there exists a relatively closed subset $S \subset \support \Vert V^\infty \Vert$ such that
\begin{enumerate}
    \item $S = \emptyset$ if $n \leq 7$,
    \item $S$ is finite if $n = 8$,
    \item $S$ has Hausdorff dimension $\leq n-8$ if $n \geq 9$, and
\end{enumerate}
and such that $\support \Vert V^\infty \Vert \setminus S$ is a smooth embedded stable minimal hypersurface in $U$. (See Theorem \ref{theo:index.0.main} for a precise statement when $n = 2$.) By the work of M. Guaraco \cite{Guaraco15}, this regularity goes through for any nonzero uniform upper bound on the Morse index when $n \geq 3$. However, it does not extend to $n = 2$ precisely because of the possible formation of singular junctions---which we study in this paper.

\section{Singularity models} \label{sec:singularity.models}

\subsection{One-dimensional solutions and De Giorgi's conjectures} \label{sec:degiorgi.conjectures} \label{sec:heteroclinic.solution} 

One-dimensional heteroclinic solutions $H : \RR \to \RR$ of \eqref{eq:pde} on $\RR$ with $\varepsilon = 1$,
\begin{equation} \label{eq:heteroclinic.solution}
	H' = \sqrt{2 W \circ H}, \; H(0) = 0,	
\end{equation}
are foundational in the theory of phase transitions. Note that there is a natural way to lift these heteroclinic solutions to higher dimensional Euclidean spaces. Namely, for any Riemannian manifold $(M^{n-1}, g)$, the function $M^n \times \RR \ni (p, t) \mapsto H(t)$, with $H : \RR \to \RR$ as in \eqref{eq:heteroclinic.solution}, is known as a ``one-dimensional'' heteroclinic solution of the Allen-Cahn equation on the product manifold $(M^{n-1} \times \RR, g + dt^2)$. Such lifts of of the heteroclinic solution play a role very similar to the one hyperplanes play in minimal surface theory. Specifically, for any choice of parameters $(\mathbf{e}, \beta) \in \SS^{n-1} \times \RR$, we get the one-dimensional entire solution
\begin{equation} \label{eq:one.dimensional.rn}
	\RR^n \ni \mathbf{x} \mapsto H(\langle \mathbf{x}, \mathbf{e} \rangle - \beta).
\end{equation}
Like planes in minimal surface theory, one-dimensional solutions stand out the most among all entire solutions to \eqref{eq:pde} in $\RR^n$ due to their simplicity and rigidity. De Giorgi conjectured \cite{DeGiorgi79} that one-dimensional solutions are the only ``monotone'' entire solutions in low enough dimensions $n$ of Euclidean space:

\begin{conj}[De Giorgi conjecture, monotone] \label{conj:degiorgi.monotone.conjecture}
	If $u : \RR^n \to \RR$ is a solution of \eqref{eq:pde} such that $\varepsilon = 1$, $|u| < 1$, $\langle \nabla u, \mathbf{e}_n \rangle > 0$, and $n \leq 8$, then $u$ is one-dimensional, i.e., of the form \eqref{eq:one.dimensional.rn}.
\end{conj}

This conjecture is inspired by the Bernstein theorem in minimal surface theory, which states that hyperplanes are the only minimal hypersurfaces that are graphical over a hyperplane $P^{n-1} \subset \RR^n$, $n \leq 8$. We discuss what is known about this conjecture in the next few paragraphs. We also point out that a closely related conjecture is:

\begin{conj}[De Giorgi conjecture, minimizers] \label{conj:degiorgi.minimizing.conjecture}
	If $u : \RR^n \to \RR$ is a solution of \eqref{eq:pde} such that $\varepsilon = 1$, $|u| < 1$, $u$ is energy minimizing among all compactly supported perturbations, and $n \leq 7$, then $u$ is one-dimensional, i.e., of the form  \eqref{eq:one.dimensional.rn}.
\end{conj}

N. Ghoussoub and C. Gui \cite{GhoussoubGui98}, confirmed Conjecture \ref{conj:degiorgi.monotone.conjecture} for $n = 2$. Their proof can be adapted to confirm Conjecture \ref{conj:degiorgi.minimizing.conjecture} The techniques of this paper offer that extension, too; see, e.g., Proposition \ref{prop:m2.characterization}.

\begin{theo}[Ghoussoub-Gui {\cite[Theorem 1.1]{GhoussoubGui98}}]  \label{theo:degiorgi.conjecture.n2}
	Conjectures \ref{conj:degiorgi.monotone.conjecture}, \ref{conj:degiorgi.minimizing.conjecture} are true when $n = 2$.
\end{theo}

L. Ambrosio and X. Cabr\'e \cite{AmbrosioCabre00} confirmed Conjecture \ref{conj:degiorgi.monotone.conjecture} for $n = 3$. Their proof, too, can be adapted to confirm Conjecture \ref{conj:degiorgi.minimizing.conjecture} in the same dimension. We refer to the work of Farina-Mari-Valdinoci \cite{FarinaMariValdinoci13} for this adaptation, which is close in spirit to the geometric measure theoretic approach in this paper.

\begin{theo}[Ambrosio-Cabr\'e {\cite[Theorem 1.2]{AmbrosioCabre00}}, Farina-Mari-Valdinoci {\cite[Theorem 1]{FarinaMariValdinoci13}}] \label{theo:degiorgi.conjecture.n3}
	Conjectures \ref{conj:degiorgi.monotone.conjecture}, \ref{conj:degiorgi.minimizing.conjecture} are true when $n = 3$. 
\end{theo}

O. Savin \cite{Savin09} confirmed the remaining cases of Conjecture \ref{conj:degiorgi.minimizing.conjecture} as well as a weaker form of the remaining cases of Conjecture \ref{conj:degiorgi.monotone.conjecture}.

\begin{theo}[Savin {\cite[Theorems 2.3, 2.4]{Savin09}}] \label{theo:degiorgi.conjecture.savin}
	Conjecture \ref{conj:degiorgi.minimizing.conjecture} is true. Conjecture \ref{conj:degiorgi.monotone.conjecture} under the additional asymptotic assumption
	\begin{equation} \label{eq:limit.at.infinity}
		\lim_{t \to \infty} u(\mathbf{x}', t) = \lim_{t \to \infty} (-u(\mathbf{x}', -t)) = 1 \text{ for all } \mathbf{x}' \in \RR^{n-1}.
	\end{equation}
\end{theo}

Both conjectures are known to fail for the standard double-well potential $W = \frac{1}{4}(1-u^2)^2$ in dimensions higher than those mentioned---as is also the case in minimal surface theory. Specifically, in the setting of phase transitions:
\begin{enumerate}
	\item Del Pino-Kowalczyk-Wei \cite{DelPinoKowalczykWei11}  constructed monotone solutions in $\RR^n$, $n \geq 9$, which are not one dimensional.
	\item Liu-Wang-Wei \cite{LiuWangWei16} constructed energy-minimizing solutions in $\RR^n$, $n \geq 8$, which are not one-dimensional.
	\item Liu-Wang-Wei's construction already yields stable counterexamples to the ``one-dimensional'' conjecture, but we also mention that Pacard-Wei \cite{PacardWei13} constructed stable solutions in $\RR^n$, $n \geq 8$, which are not one-dimensional.
\end{enumerate}

In closing, we mention the following result, which serves as a parallel to Allard's regularity theorem from minimal surface theory. It follows in a straightforward manner from the work of K. Wang \cite{Wang14}, where he also provided an alternate proof of Savin's theorem.

\begin{theo}[cf. Wang {\cite[Theorem 9.1]{Wang14}}] \label{theo:wang.allard}
	Suppose $u : \RR^n \to \RR$ is a solution of \eqref{eq:pde} with $\varepsilon = 1$, $|u| < 1$, and
	\[ \lim_{R \uparrow \infty} \frac{(E_1 \restr B_R(\mathbf{0}))[u]}{\energyunit R^{n-1}} = 1. \]
	Then $u$ is one-dimensional.
\end{theo}

We provide a proof of this Theorem, as it is not explicitly written down in the literature in a way that is directly applicable.

\begin{proof}[Proof of Theorem \ref{theo:wang.allard}]
	Without loss of generality, after precomposing with a rigid motion, we may assume that $u(\mathbf{0}) = 0$. Let $\varepsilon_A$, $\tau_A$, $\alpha_A \in (0, 1)$, $R_A$, $K_A > 0$ be as in \cite[Theorem 9.1]{Wang14}. By construction, 
	\begin{equation} \label{eq:m2.characterization.i}
		\lim_{R \uparrow \infty} \energyunit^{-1} \omega_1^{-1} R^{-1} (E \restr B_R(\mathbf{0}))[u] = 1.
	\end{equation}
	Therefore, there exists $R_0 \geq \varepsilon_A^{-1} R_A$ such that, for all $R \geq R_0$, the rescaled function
	\[ u_R(\mathbf{x}) \triangleq u \left( \frac{R\mathbf{x}}{R_A} \right), \; \mathbf{x} \in B_R(\mathbf{0}), \]
	satisfies
	\[ h_0^{-1} \omega_1^{-1} R^{-1} (E \restr B_R(\mathbf{0}))[u] \leq 1+\tau_A, \]
	\[ \Delta u_R(\mathbf{x}) = \left( \frac{R}{R_A} \right)^2 W'(u) \text{ in } B_R(\mathbf{0}), \]
	and $\varepsilon \triangleq \frac{R_A}{R} \leq \varepsilon_A$. Invoking \cite[Theorem 9.1]{Wang14}, and passing to a subsequence $R_j \uparrow \infty$, there exists a fixed hyperplane $P \subset \RR^n$, such that, for all $t \in [-1/2,1/2]$, $j = 1, 2, \ldots$,
	\begin{multline*}
		\{ u_{R_j} = t \} \cap \{ (\mathbf{x}',x_n) \in \RR^{n-1} \times \RR : |\mathbf{x}'| \leq 1 \} \\
		= \{ (\mathbf{x}', h^t_j(\mathbf{x}')) : |\mathbf{x}'| \leq 1 \},
	\end{multline*}
	where $h_j^t : \closure{B}_1^{n-1}(\mathbf{0}) \to \RR$ is such that $|h_j^t|_{C^{1,\alpha_A}(B_1)} \leq K_A$. Undoing the scaling, this implies that for all $t \in [-1/2,1/2]$, $j = 1, 2, \ldots$,
	\begin{multline*}
		\{ u = t \} \cap \left\{ (\mathbf{x}', x_n) \in \mathbf{R}^{n-1} \times \RR : |\mathbf{x}'| \leq \frac{R_j}{R_A} \right\} \\
		= \left\{ (\mathbf{x}', h^t(\mathbf{x}')) : |\mathbf{x}'| \leq \frac{R_j}{R_A} \right\},
	\end{multline*}
	where $h^t$ is a $C^{1,\alpha_A}$ function, whose gradient satisfies the scale-invariant estimate
	\[ \left( \frac{R}{R_A} \right)^{\alpha_A} [\nabla h^t]_{C^{\alpha_A}(B_{R_j/R_A}(\mathbf{0}))} \leq K_A. \]
	Letting $j \uparrow \infty$, we conclude that $h^t$ is constant for each $t \in [-1/2,1/2]$. Thus, $u$ is one-dimensional by unique continuation.
\end{proof}

\subsection{Moduli space $\cM_{2k}$ of $2k$-ended solutions in $\RR^2$} \label{sec:m2k}

M. del Pino, M. Kowalczyk, and F. Pacard defined in \cite{DelPinoKowalczykPacard13} a space of solutions of \eqref{eq:pde} with $\varepsilon = 1$ on $\RR^2$, $\cM_{2k}$, that looks from infinity like a collection of $2k$ copies of the one-dimensional heteroclinic solution. We recall the construction of this space here (after \cite{DelPinoKowalczykPacard13}) for the sake of completeness.

Fix $k \in \{1, 2, \ldots \}$. We denote by $\Lambda^{2k}$ (denoted $\Lambda^{2k}_{\operatorname{ord}}$ in \cite{DelPinoKowalczykPacard13}) the space of ordered $2k$-tuples $\lambda = (\lambda_1, \ldots, \lambda_{2k})$ of oriented affine lines on $\RR^2$, parametrized as
\[ \lambda_j = (r_j, \mathbf{f}_j) \in \RR \times \SS^1, \; j = 1, \ldots, 2k, \]
where $\mathbf{f}_j = (\cos \theta_j, \sin \theta_j)$, and
\[ \theta_1 < \ldots < \theta_{2k} < 2\pi + \theta_1. \]
For $\lambda \in \Lambda^{2k}$, we denote
\begin{equation} \label{eq:m2k.minimum.angle}
	\theta_\lambda \triangleq \frac{1}{2} \min \{ \theta_2 - \theta_1, \ldots, \theta_{2k} - \theta_{2k-1}, 2\pi + \theta_1 - \theta_{2k} \}.	
\end{equation}
Fix $\lambda \in \Lambda^{2k}$. For large $R > 0$ and all $j = 1, \ldots, 2k$, there exists $s_j \in \RR$ such that $r_j \JJ \mathbf{f}_j + s_j \mathbf{f}_j \in \partial B_R(\mathbf{0})$, the half-lines $\lambda_j^+ \triangleq r_j \JJ \mathbf{f}_j + s_j \mathbf{f}_j + \RR_+ \mathbf{f}_j$ are disjoint and contained in $\RR^2 \setminus B_R(\mathbf{0})$, and the minimum distance of any two distinct $\lambda_i^+$, $\lambda_j^+$ is $\geq 4$. (Here, $\JJ \in \operatorname{End}(\RR^2)$ is the counterclockwise rotation map by $\frac{\pi}{2}$.) The affine half-lines $\lambda_1^+, \ldots, \lambda_{2k}^+$ and the circle $\partial B_R(\mathbf{0})$ induce a decomposition of $\RR^2$ into $2k+1$ open sets,
\begin{multline} \label{eq:m2k.r2.decomposition}
	\Omega_0 \triangleq B_{R+1}(\mathbf{0}) \text{, and} \\
	\Omega_j \triangleq \bigcap_{i \neq j} \{ \mathbf{x} \in \RR^2 \setminus B_{R-1}(\mathbf{0}) : \dist(\mathbf{x}, \lambda_j^+) < \dist(\mathbf{x}, \lambda_i^+) + 2 \}, \\
	j = 1, \ldots, 2k
\end{multline}
Note that these open sets are not disjoint. Then, we define $\chi_{\Omega_0}, \ldots, \chi_{\Omega_{2k}}$ to be a smooth partition of unity of $\RR^2$ subordinate to $\Omega_0, \ldots, \Omega_{2k}$, and such that
\begin{multline} \label{eq:m2k.r2.decomposition.disjoint}
	\chi_{\Omega_0} \equiv 1 \text{ on } \Omega_0' \triangleq B_{R-1}(\mathbf{0}), \text{ and } \\
	\chi_{\Omega_j} \equiv 1 \text{ on } \Omega_j' \triangleq \cap_{i\neq j} \{ \mathbf{x} \in \RR^2 \setminus B_{R+1}(\mathbf{0}) : \dist(\mathbf{x}, \lambda_j^+) < \dist(\mathbf{x}, \lambda_i^+) - 2\}, \\
	j = 1, \ldots, 2k.
\end{multline}
Note that these new open sets are disjoint. Without loss of generality, $|\chi_{\Omega_j}| + |\nabla \chi_{\Omega_j}| + |\nabla^2 \chi_{\Omega_j}| \leq c_1$ for all $j = 0, \ldots, 2k$, with $c_1 = c_1(\theta_\lambda)$. Finally, we define
\begin{equation} \label{eq:m2k.approximate.solution}
	u_\lambda \triangleq \sum_{j=1}^{2k} (-1)^{j+1} \chi_{\Omega_j} H(\dist^s(\cdot, \lambda_j)),	
\end{equation}
where $\dist^s(\cdot, \lambda_j)$ denotes the signed distance to $\lambda_j$, taking $\JJ\mathbf{f}_j$ to be the positive direction. Here, $H$ is the heteroclinic solution \eqref{eq:heteroclinic.solution}.

\begin{defi}[Del Pino-Kowalczyk-Pacard {\cite[Definition 2.2]{DelPinoKowalczykPacard13}}] \label{defi:m2k}
	For $k \geq 1$, we denote
	\begin{equation} \label{eq:m2k.s2k}
		\cS_{2k} \triangleq \bigcup_{\lambda \in \Lambda^{2k}} \{ u \in C^\infty(\RR^2) : u - u_\lambda \in W^{2,2}(\RR^2) \}.
	\end{equation}
	We endow $\cS_{2k}$ with the weak topology of the operator
	\begin{equation} \label{eq:m2k.J}
		\cJ : \cS_{2k} \to W^{2,2}(\RR^2) \times \Lambda^{2k}, \; \cJ(u) \triangleq (u-u_\lambda, \lambda).
	\end{equation}
	Finally, we define the space of ``$2k$-ended solutions'' to be
	\begin{equation} \label{eq:m2k}
		\cM_{2k} \triangleq \{ u \in \cS_{2k} \text{ satisfying } \eqref{eq:pde} \text{ with } \varepsilon = 1 \}.
	\end{equation}
\end{defi}

\begin{exam} \label{exam:m2.one.dimensional}
	Elements of $\cM_2$ are the lifts to $\RR^2$ of one-dimensional heteroclinic solutions \eqref{eq:heteroclinic.solution},
	\[ \RR^2 \ni \mathbf{x} \mapsto H(\langle \mathbf{x}, \mathbf{e} \rangle - \beta) \text{, with } (\mathbf{e}, \beta) \in \SS^1 \times \RR. \]
	See Theorem \ref{theo:wang.allard} and/or Proposition \ref{prop:m2.characterization}.
\end{exam}

The following result, due to M. del Pino, M. Kowalczyk, and F. Pacard, significantly improves the a priori $W^{2,2}$ decay of $u - u_\lambda$ to an exponential decay:

\begin{theo}[Del Pino-Kowalczyk-Pacard {\cite[Theorem 2.1]{DelPinoKowalczykPacard13}}] \label{theo:m2k.refined.asymptotics}
	Let $u_0 \in \cM_{2k}$. There exists a neighborhood $U \subset \cM_{2k}$ and a $\delta = \delta(u_0) > 0$ such that
	\begin{equation} \label{eq:m2k.refined.asymptotics.i}
		\cJ(U) \subseteq e^{-\delta |\mathbf{x}|} W^{2,2}(\RR^2) \times \Lambda^{2k},
	\end{equation}
	and, moreover, such that the restricted map
	\begin{equation} \label{eq:m2k.refined.asymptotics.ii}
		\cJ|_U : U \to e^{-\delta |\mathbf{x}|} W^{2,2}(\RR^2) \times \Lambda^{2k}	
	\end{equation}
	is continuous with respect to the corresponding topologies; here, $\cJ$ is the map defined in \eqref{eq:m2k.J}.
\end{theo}

We conclude this section by remarking that $\cM_{2k}$ of $2k$-ended solutions in $\RR^2$ exhausts all finite Morse index solutions with linear energy growth, which correspond precisely to our desired singularity models in $\RR^2$. The ``$\subseteq$'' direction of \eqref{eq:finite.index.equivalent.m2k} is precisely \cite[Theorem 2.8]{KowalczykLiuPacard12}. The ``$\supseteq$'' direction essentially follows from the work of K. Wang in \cite{Wang15}; we include here the necessary argument that transports one from the setting of \cite{Wang15} to that of $\cM_{2k}$.

\begin{prop} \label{prop:finite.index.equivalent.m2k}
	The following equality of sets holds true
	\begin{multline} \label{eq:finite.index.equivalent.m2k}
		\bigsqcup_{k=1}^\infty \cM_{2k} = \Big\{ u \in C^\infty_{\loc}(\RR^2) \cap L^\infty(\RR^2) \text{ satisfying } \eqref{eq:pde} \text{ with } \varepsilon = 1, \\
		\ind(u) < \infty, \text{ and } 0 < \limsup_{R \uparrow \infty} R^{-1} (E_1 \restr B_1)[u] < \infty \Big\}.
	\end{multline}
\end{prop}

\begin{rema} \label{rema:finite.index.conjecture}
	In the setting of this paper we only need to study singularity models that arise from blowups with linear energy growth, so the energy assumption holds true. However, Wang-Wei have recently announced \cite{WangWei17} that the 
	\[ \lim_{R \uparrow \infty} R^{-1} (E_1 \restr B_1)[u] < \infty \]
	assumption above is entirely unnecessary, showing it is automatically true whenever $\ind(u) < \infty$. Their proof makes use of strong curvature estimates, which they derive, and which we will also need in our study of singularity formation; see Sections \ref{sec:wang.wei.curvature.estimates}, \ref{sec:index.1.singularity.formation}.
\end{rema}

\begin{proof}[Proof of Proposition \ref{prop:finite.index.equivalent.m2k}]
	By virtue of \cite[Theorems 1.2, 1.3]{Wang15}, we know that there exist $k = k(u) \geq 1$, $R = R(u) > 1$, disjoint, embedded curves $\Gamma_1, \ldots, \Gamma_{2k}$, and angles
	\[ \underline{\varphi}_1 < \overline{\varphi}_1 = \underline{\varphi}_2 < \ldots < \overline{\varphi}_{2k-1} = \underline{\varphi}_{2k} < \overline{\varphi}_{2k} = 2\pi + \underline{\varphi}_1 \]
	such that
	\begin{equation} \label{eq:finite.index.equivalent.m2k.i}
		\{ u = 0 \} \setminus B_R(\mathbf{0}) = \sqcup_{i=1}^{2k} \Gamma_i,
	\end{equation}
	where
	\[ \Gamma_i \subset S_i \triangleq \left\{ r \mathbf{f}(\theta) : r \geq R \text{ and } \theta \in (\underline{\varphi}_i, \overline{\varphi}_i) \right\} \text{ for } i = 1, \ldots, 2k \text{;} \]
	here, $\mathbf{f}(\theta) = (\cos \theta, \sin \theta) \in \RR^2$.	Following the argument in \cite[Theorem 1.3]{Wang15}, we can write each $\Gamma_i$ as a smooth graph over a ray $\rho_i \triangleq \left\{ r \mathbf{f}(\theta_i) : r \geq R \right\}$ with $\theta_i \in (\underline{\varphi}_i, \overline{\varphi}_i)$, after possibly enlarging $R$; i.e.,
	\[ \Gamma_i = \left\{ r \mathbf{f}(\theta_i) + h_i(r) \JJ \mathbf{f}(\theta_i) : r \geq R \right\} \]
	with $h_i : C^\infty([R, \infty))$, for all $i = 1, \ldots, 2k$. We also have
	\begin{equation} \label{eq:finite.index.angle.sum.zero}
		\sum_{i=1}^{2k} \mathbf{f}(\theta_i) = \mathbf{0} \in \RR^2	
	\end{equation}
	by \cite[Theorem 1.1 (v)]{Wang15}. Next, \cite[Theorem 3.3]{Wang15} implies
	\begin{equation} \label{eq:finite.index.equivalent.m2k.ii}
		|h_i(r) - \tau_i| \leq C_0 e^{-C_0^{-1} r}
	\end{equation}
	for some $\tau_i = \tau_i(u) \in \RR$ and $C_0 = C_0(u) > 0$, and, up to a possible change of sign, that
	\begin{equation} \label{eq:finite.index.concave.sup.estimate}
		\left| u(r \mathbf{f}(\theta)) - (-1)^{i+1} H(r \sin (\theta - \theta_i) - \tau_i) \right| \leq C_1 e^{-C_1^{-1} r}
	\end{equation}
	for some $C_1 = C_1(u) > 0$ and for all $r \geq R$, $\theta \in (\underline{\varphi}_i, \overline{\varphi}_i)$, $i = 1, \ldots, 2k$. (``Up to a possible change of sign'' means that $(-1)^{i+1}$ may have to be replaced by a $(-1)^i$.)

	From elliptic regularity, \eqref{eq:finite.index.concave.sup.estimate} readily implies
	\begin{equation} \label{eq:finite.index.concave.c1.estimate}
		\left| \nabla u(r \mathbf{f}(\theta)) - (-1)^i H'(r \sin (\theta - \theta_i) - \tau_i) \JJ \mathbf{f}(\theta_i) \right| \leq C_2 e^{-C_2^{-1} r}		
	\end{equation}
	and
	\begin{multline} \label{eq:finite.index.concave.c2.estimate}
		\left| \nabla^2 u( r \mathbf{f}(\theta)) - (-1)^i H''(r \sin (\theta - \theta_i) - \tau_i) \JJ \mathbf{f}(\theta_i) \otimes \JJ \mathbf{f}(\theta_i) \right| \\
		\leq C_2 e^{-C_2^{-1}r},	
	\end{multline}
	for $C_2 = C_2(u, W)$, and all $r \geq R$, $\theta \in (\underline{\varphi}_i, \overline{\varphi}_i)$, $i = 1, \ldots, 2k$. From \cite[Proposition 2.1]{Gui12}, for any $0 < \varepsilon < \frac{1}{2} \min_{i=1,\ldots,2k} \{ \overline{\varphi}_i - \theta_i, \theta_i - \underline{\varphi}_i \}$, \eqref{eq:finite.index.equivalent.m2k.i}, and \eqref{eq:finite.index.equivalent.m2k.ii},
	\begin{multline} \label{eq:finite.index.convex.estimate}
		\sup \left\{ \sum_{\ell=0}^2 \left| \nabla^\ell u (r \mathbf{f}(\theta)) \right| : r \geq R, |\theta-\theta_i| \geq \varepsilon \text{ for all } i = 1, \ldots, 2k \right\} \\
		\leq C_3 e^{-C_3^{-1}r},
	\end{multline}
	where $C_3 = C_3(W, \varepsilon, u)$.
	It follows from \eqref{eq:finite.index.concave.sup.estimate}, \eqref{eq:finite.index.concave.c1.estimate}, \eqref{eq:finite.index.concave.c2.estimate}, and  \eqref{eq:finite.index.convex.estimate}, that $u - u_\lambda \in W^{2,2}$ if $u_\lambda$ is an approximate $2k$-ended solution with $\lambda = (\lambda_1, \ldots, \lambda_{2k}) \in \Lambda^{2k}$ given by $\lambda_i \triangleq (\tau_i, \mathbf{f}(\theta_i))$, $i = 1, \ldots, 2k$. The result follows.
\end{proof}

\subsection{Effects of topology at infinity on the Morse index in $\cM_{2k}$} \label{sec:m2k.index.lower.bound}

In this section we prove Theorem \ref{theo:m2k.index.lower.bound}, which relates the Morse index of $2k$-ended solutions to \eqref{eq:pde} to their structure at infinity. We assume $k \geq 2$, because elements of $\cM_2$ are all known to be stable (see Theorem \ref{theo:wang.allard} and/or Proposition \ref{prop:m2.characterization}).

To prove this theorem, we will need to obtain a precise pointwise understanding of kernel elements of the Jacobi operator, seeing as to how they will play a significant role in the relevant variational theory:

\begin{defi} \label{defi:jacobi.fields}
	If $u$ is a critical point of $E$ in $U$, then the space of its Jacobi fields consists of all functions $v$ that satisfy $- \Delta v + W''(u) v = 0$ in $U$ in the classical sense.
\end{defi}

Denote $R$, $\lambda \in \Lambda^{2k}$, and $u_\lambda$ the objects associated with $u$ by its construction as an element of $\cM_{2k}$ in Section \ref{sec:m2k}. Also, denote $\lambda = (\lambda_1, \ldots, \lambda_{2k})$, with $\lambda_i = (\tau_i, \mathbf{f}_i) \in \RR \times \SS^1$. Recall from \cite[(2.16)]{DelPinoKowalczykPacard13} that:
\begin{equation} \label{eq:m2k.index.lower.bound.i}
	\sum_{i=1}^{2k} \mathbf{f}_i = \mathbf{0},	
\end{equation}
and that, after possibly enlarging $R > 0$, $\{ u = 0 \} \setminus B_R(\mathbf{0})$ decomposes into $2k$ disjoint curves $\Gamma_i$, $i = 1, \ldots, 2k$, and, for some $\delta < \theta_\lambda(u)$, $\Gamma_i \subset S(\mathbf{f}_i, \delta/2, R)$ with $S(\mathbf{f}_i, \delta, R)$ all pairwise disjoint. Here,
\begin{equation} \label{eq:m2k.index.lower.bound.ii}
	S(\mathbf{e}, \theta, R) \triangleq \{ r \mathbf{f} : r \geq R, \dist_{\SS^1}(\mathbf{f}, \mathbf{e}) < \theta \}.
\end{equation}
Finally, using Theorem \ref{theo:m2k.refined.asymptotics} we see that, perhaps after shrinking $\delta > 0$ and enlarging $R > 0$, and perhaps after an ambient rigid motion,
\begin{equation} \label{eq:m2k.index.lower.bound.iii}
	\frac{\nabla u}{|\nabla u|} \approx (-1)^{i+1} \JJ \mathbf{f}_i \text{ in } S(\mathbf{f}_i, \delta/2, R).
\end{equation}
		
Lay out $\mathbf{f}_1, \ldots, \mathbf{f}_{2k} \in \SS^1$, and color them red (negative) or blue (positive) depending on the sign of $\langle (-1)^{i+1} \JJ \mathbf{f}_i, \mathbf{e} \rangle$. Here, $\mathbf{e} \in \SS^1$ is a fixed direction, chosen generically, so that $\langle \JJ \mathbf{f}_i, \mathbf{e} \rangle \neq 0 \text{ for all } i = 1, \ldots, 2k$. We will temporarily need the following generalization of $\JJ$:
\[ \JJ_\theta \in \operatorname{End}(\RR^2) \text{ acting by } \begin{bmatrix} \cos \theta & -\sin \theta \\ \sin \theta & \cos \theta \end{bmatrix}. \]
There exist unique $\varphi_1, \ldots, \varphi_{2k} \in (0, 2\pi)$ such that $\mathbf{f}_{i+1} = \JJ_{\varphi_i}(\mathbf{f}_i)$ for all $i = 1, \ldots, 2k$. It's easy to see that
\begin{equation} \label{eq:m2k.index.lower.bound.iv}
	\varphi_i \in (0, \pi) \text{ for all }	 i = 1, \ldots, 2k
\end{equation}
by combining \eqref{eq:m2k.index.lower.bound.i} with $k \geq 2$ (recall that we're assuming $k \geq 2$).

\begin{clai}
	If $\mathbf{f}_{2\ell-1}$, $\mathbf{f}_{2\ell}$ have the same color, blue, then
	\[ \JJ_{-\varphi_{2\ell-1}} \mathbf{e}, \mathbf{f}_{2\ell-1}, \mathbf{e} \text{ lie counterclockwise on } \SS^1 \text{ in the order listed;} \]
	else, if their common color is red, then
	\[ \JJ_{-\varphi_{2\ell-1}} (-\mathbf{e}), \mathbf{f}_{2\ell-1}, -\mathbf{e} \text{ lie counterclockwise  on } \SS^1 \text{ in the order listed.} \] 
\end{clai}
\begin{proof}
	Without loss of generality, we assume $\ell = 1$. Recall that the respective colors are determined by the signs of $\langle \JJ_{\frac{\pi}{2}} \mathbf{f}_{1}, \mathbf{e} \rangle$ and $\langle -\JJ_{\frac{\pi}{2}} \mathbf{f}_{2}, \mathbf{e} \rangle = \langle \JJ_{-\frac{\pi}{2}+\varphi_{1}} \mathbf{f}_{1}, \mathbf{e} \rangle$.
	
	Denote $\cP \triangleq \{ \mathbf{f} \in \SS^1 : \langle \mathbf{f}, \mathbf{e} \rangle > 0 \}$. 	If both colors are blue, then
	\[ \JJ_{\frac{\pi}{2}} \mathbf{f}_{1} \in \cP \iff \mathbf{f}_{1} \in \JJ_{-\frac{\pi}{2}}(\cP) \]
	and
	\[ \JJ_{-\frac{\pi}{2}+\varphi_{1}} \mathbf{f}_{1} \in \cP \iff \mathbf{f}_{1} \in \JJ_{\frac{\pi}{2} - \varphi_{1}}(\cP), \]
	i.e.,
	\[ \mathbf{f}_{1} \in \JJ_{-\frac{\pi}{2}}(\cP) \cap \JJ_{\frac{\pi}{2} - \varphi_{1}}(\cP). \]
	Using \eqref{eq:m2k.index.lower.bound.iv}, we see that the three vertices $\JJ_{-\varphi_{1}} \mathbf{e}$, $\mathbf{f}_{1}$, and $\mathbf{e}$, must lie counterclockwise in this order on $\SS^1$.
	
	If both colors are red, then, by a similar argument,
	\[ \mathbf{f}_1 \in \JJ_{-\frac{\pi}{2}}(-\cP) \cap \JJ_{\frac{\pi}{2}-\varphi_1}(-\cP), \]
	and we see that the three vertices $\JJ_{-\varphi_1}(-\mathbf{e})$, $\mathbf{f}_1$, and $-\mathbf{e}$, must lie counterclockwise in this order on $\SS^1$. 
\end{proof}

In a completely analogous manner, one also checks that:

\begin{clai}
	If $\mathbf{f}_{2\ell}$, $\mathbf{f}_{2\ell+1}$ have the same color, blue, then
	\[ \JJ_{-\varphi_{2\ell}} (-\mathbf{e}), \mathbf{f}_{2\ell}, -\mathbf{e} \text{ lie counterclockwise  on } \SS^1 \text{ in the order listed;} \] 
	else, if their common color is red, then
	\[ \JJ_{-\varphi_{2\ell}} \mathbf{e}, \mathbf{f}_{2\ell}, \mathbf{e} \text{ lie counterclockwise on } \SS^1 \text{ in the order listed.} \]
\end{clai}

We now make the following key observation:

\begin{clai}
	There exist at least $2k-2$ groups of consecutive same-colored vertices.
\end{clai}
\begin{proof}[Proof of claim]
	Within the space of valid colorings,
	\begin{equation} \label{eq:m2k.index.lower.bound.v}
		\{\text{existence of blue } \mathbf{f}_{2\ell-1}, \mathbf{f}_{2\ell} \} \cap \{ \text{existence of red } \mathbf{f}_{2m}, \mathbf{f}_{2m+1} \} = \emptyset.
	\end{equation}
	This follows by combining the previous two claims. Likewise
	\begin{equation} \label{eq:m2k.index.lower.bound.vi}
		\{\text{existence of red } \mathbf{f}_{2\ell-1}, \mathbf{f}_{2\ell} \} \cap \{ \text{existence of blue } \mathbf{f}_{2m}, \mathbf{f}_{2m+1} \} = \emptyset.
	\end{equation}
	There are now the following cases to consider:
	\begin{enumerate}
		\item There exist three consecutive same-colored vertices. Then, by combining the previous two claims and engaging in elementary angle-chasing, it follows that there do not exist any more consecutive same-colored vertices. In this case, it follows that there are precisely $2k-2$ groups of consecutive same-colored vertices.
		\item There are no three consecutive same-colored vertices. Then, together with \eqref{eq:m2k.index.lower.bound.v}, \eqref{eq:m2k.index.lower.bound.vi}, it follows that there are at least $2k-2$ groups of consecutive same-colored vertices.
	\end{enumerate}
	This concludes the proof of the claim.
\end{proof}

Given this claim, differentiate \eqref{eq:pde} in the direction of $\mathbf{e} \in \SS^1$. We see that $v \triangleq \langle \nabla u, \mathbf{e} \rangle$ satisfies
\begin{equation} \label{eq:m2k.index.lower.bound.vii}
	\Delta v = W''(u) v \text{ in } \RR^2.	
\end{equation}

Define
\begin{align*}
	\cN & \triangleq \{ v = 0 \} \text{ (the ``nodal set''),} \\
	\cS & \triangleq \cN \cap \{ \nabla v = \mathbf{0} \} \text{ (the ``singular set'').}
\end{align*}
By the implicit function theorem, $\cN \setminus \cS$ consists of smooth, injectively immersed curves in $\RR^2$. By Bers' theorem (see, e.g., \cite{BersJohnSchechter79}), $\cS$ consists of at most countably many points and, for each $p \in \cS$, there exists $r = r(p)$ such that, up to a diffeomorphism of $B_r(p)$,
\begin{multline} \label{eq:m2k.index.lower.bound.viii}
	\cN \cap B_r(p) \approx \text{the zero set of a} \\
	\text{homogeneous even-degree harmonic polynomial.}
\end{multline}

Denote $\Omega_1, \ldots, \Omega_q \subset \RR^2 \setminus \cN$ the nodal domains (i.e., connected components of $\{ v \neq 0 \}$), labeling so that $\Omega_1, \ldots, \Omega_p$ are the unbounded ones, and $\Omega_{p+1}, \ldots, \Omega_q$ are the bounded ones. By virtue of our precise understanding of $\cN$, $\cS$, as discussed above, we know that they are all open, connected, Lipschitz domains.

\begin{rema} \label{rema:m2k.index.lower.bound.q.finite}
	The notation used here implicitly asserts that there are finitely many nodal domains. This follows, a posteriori, by the proof of the following claim and \cite[Theorem 2.8]{KowalczykLiuPacard12}.
\end{rema}

\begin{clai}
	$\ind(u) \geq q-1$.
\end{clai}

\begin{proof}[Proof of claim]
	First, it's standard that for every bounded nodal domain $\Omega_{p+1}, \ldots, \Omega_q$ we have
	\begin{equation} \label{eq:m2k.index.lower.bound.ix}
		\nul(u; \Omega_i) \geq 1 \text{ for all } i = p+1, \ldots, q.
	\end{equation}
	Now we move on to unbounded nodal domains. It is not hard to see that we have at least two such. Suppose that $\Omega_1$ is an unbounded nodal domain, and suppose $\Omega_2$ is its counterclockwise neighboring unbounded nodal domain. By \eqref{eq:m2k.index.lower.bound.viii}, $v$ attains opposite signs on $\Omega_1$, $\Omega_2$. Thus, $v$ is a bounded, sign-changing Jacobi field in $\Omega_{12} \triangleq \intr{\closure{\Omega}_1 \cup \closure{\Omega}_2}$, which is itself an open, connected, unbounded Lischitz domain. By Lemma \ref{lemm:app.pde.unstable.past.nodal.domain}, $\ind(u; \Omega_{12}) \geq 1$. Since $\Omega_{12}$ is unbounded, we have
	\[ \ind(u; \widetilde{\Omega}_2) \geq 1 \text{ for some bounded } \widetilde{\Omega}_2 \subsetneq \Omega_{12} \]	
	which is itself open, connected, and Lipschitz. Denote $\widetilde{\Omega}_1 = \emptyset$.
	
	Proceeding similarly (and labeling accordingly) in the counterclockwise direction, we can construct disjoint, bounded, open, connected, Lipschitz $\widetilde{\Omega}_3, \ldots, \widetilde{\Omega}_p$, 
	\begin{equation} \label{eq:m2k.index.lower.bound.x}
		\ind(u; \widetilde{\Omega}_i) \geq 1 \text{ for all } i = 2, \ldots, p,
	\end{equation}
	where $\widetilde{\Omega}_i \subset \intr{(\closure{\Omega}_{i-1} \cup \closure{\Omega}_i)}$. More precisely, at each stage $i$ we have to sacrifice a bounded portion of $\closure{\Omega}_1 \cup \cdots \cup \closure{\Omega}_{i-1} \setminus (\widetilde{\Omega}_1 \cup \cdots \cup \widetilde{\Omega}_{i-1})$ to give rise to a negative eigenvalue on a slight enlargement of $\Omega_i$, which is bounded and disjoint from $\widetilde{\Omega}_1 \cup \cdots \cup \widetilde{\Omega}_{i-1}$.
	
	The claim follows by combining \eqref{eq:m2k.index.lower.bound.ix}, \eqref{eq:m2k.index.lower.bound.x}, and Theorem \ref{theo:app.pde.courant.nodal.domain}.
\end{proof}

We now estimate $q-1$ from below. It will be convenient to assume that $\cS$ and the set of connected components of $\cN \setminus \cS$ are both finite sets---refer to Remark \ref{rema:m2k.index.lower.bound.infinities} for the minor necessary adjustments to deal with the general case. From Euler's formula for planar graphs, we know that
\begin{equation} \label{eq:m2k.index.lower.bound.xi}
	q = 1 + |\{ \text{connected components of } \cN \setminus \cS \}| - |\cS|,
\end{equation}
where $|\cdot|$ denotes the cardinality of a set. By \eqref{eq:m2k.index.lower.bound.viii}, every connected component $\Gamma$ of $\cN \setminus \cS$ is a smooth curve with
\begin{equation} \label{eq:m2k.index.lower.bound.xii}
	|\partial \Gamma| = 0, 1, \text{ or } 2,
\end{equation}
depending on whether $\Gamma$ is infinite in both directions, one direction, or is finite. Counting the set of pairs $(v, e)$ of vertices and edges in $\cN$ in two ways, we see that

\begin{clai}
	$q \geq k$.
\end{clai}
\begin{proof}
	The fact that there exist at least $2k-2$ groups of consecutive same-colored vertices implies that there exists $R > 0$ sufficiently large so that $\cS \subset B_R(\mathbf{0})$ and $\cN \setminus B_R(\mathbf{0})$ has at least $2k-2$ components. By a straightforward counting argument combined with \eqref{eq:m2k.index.lower.bound.xii}, this implies
	\begin{equation} \label{eq:m2k.index.lower.bound.xiii}
		2k-2 \leq \sum_{\ell=0}^2 (2-\ell) \cdot |\{ \text{connected components } \Gamma \subset \cN \setminus \cS : |\partial \Gamma| = \ell \}|.
	\end{equation}
	On the other hand, by counting the elements of the set
	\[ \cA \triangleq \{ (p, \Gamma) : p \in \cS, \; \Gamma =  \text{connected component of } \cN \setminus \cS \text{ incident to } p \} \]
	in one way, we find that
	\begin{equation} \label{eq:m2k.index.lower.bound.xiv}
		|\cA| = \sum_{\ell=0}^2 \ell \cdot |\{ \text{connected components } \Gamma \subset \cN \setminus \cS : |\partial \Gamma| = \ell \}|.
	\end{equation}
	Adding \eqref{eq:m2k.index.lower.bound.xiii}, \eqref{eq:m2k.index.lower.bound.xiv}, and rearranging, we get
	\begin{multline} \label{eq:m2k.index.lower.bound.xv}
		2 \cdot |\{ \text{connected components of } \cN \setminus \cS \}| \geq |\cA| + 2k-2 \\
		\iff |\{ \text{connected components of } \cN \setminus \cS \}| \geq \frac{1}{2} |\cA| + k-1.
	\end{multline}
	Plugging \eqref{eq:m2k.index.lower.bound.xv} into \eqref{eq:m2k.index.lower.bound.xi} yields the estimate
	\begin{equation} \label{eq:m2k.index.lower.bound.xvi}
		q \geq k + \frac{1}{2} |\cA| - |\cS|.
	\end{equation}
	On the other hand, because of \eqref{eq:m2k.index.lower.bound.viii}, each $p \in \cS$ contributes at least two elements to $\cA$; i.e., $|\cA| \geq 2 \cdot |\cS|$. The claim follows.
\end{proof}

\begin{rema} \label{rema:m2k.index.lower.bound.infinities}
	The proof above assumed that
	\[ |\cS| + |\{ \text{connected components of } \cN \setminus \cS \}| < \infty, \]
	so let us discuss the necessary adjustments for it to go through in the general case. By the finiteness of $q$ (see Remark \ref{rema:m2k.index.lower.bound.q.finite}), we know that there exists a large enough radius $R$ so that $\Omega_i \cap B_R(\mathbf{0})$ is connected and nonempty for every $i = 1, \ldots, q$. By the local finiteness of $\cS$, we may further arrange for $\partial B_R(\mathbf{0}) \cap \cS = \emptyset$ and for all intersections $\partial \Omega_i \cap \partial B_R(\mathbf{0})$, $i = 1, \ldots, 2k$, to be transverse. The finite planar graph arrangement contained within $B_R(\mathbf{0})$ has the same number of faces as the original infinite planar graph arrangement. We may, therefore, repeat the previous proof, starting at Remark \ref{rema:m2k.index.lower.bound.q.finite}, discarding all elements of $\cS$ and components of $\cN \setminus \cS$ that lie fully outside of $B_R(\mathbf{0})$, and identifying $\partial B_R(\mathbf{0})$ with infinity.
\end{rema}

Combining everything above, we obtain the thesis of Theorem \ref{theo:m2k.index.lower.bound}.

There is a finer characterization of $\cM_2$ in terms of Morse index than the one in Theorem \ref{theo:m2k.index.lower.bound}:

\begin{prop} \label{prop:m2.characterization}
	The following are all equivalent:
	\begin{enumerate}
		\item $u \in \cM_2$;
		\item $u(\mathbf{x}) \equiv H(\langle \mathbf{x}, \mathbf{e} \rangle - \beta)$ where $H$ is as in \eqref{eq:heteroclinic.solution}, and $(\mathbf{e}, \beta) \in \SS^1 \times \RR$;
		\item $u \in C^\infty_{\loc}(\RR^2) \cap L^\infty(\RR^2)$ satisfies \eqref{eq:pde}, with $\varepsilon = 1$ and $\langle \nabla u, \mathbf{e} \rangle > 0$ for some fixed $\mathbf{e} \in \SS^1$;
		\item $u \in C^\infty_{\loc}(\RR^2) \cap L^\infty(\RR^2)$ is a nonconstant minimizer of the energy in \eqref{eq:energy}, with $\varepsilon = 1$, among compact perturbations;
		\item $u \in C^\infty_{\loc}(\RR^2) \cap L^\infty(\RR^2)$ is a nonconstant stable critical point of the energy in \eqref{eq:energy}, with $\varepsilon = 1$; i.e., $\ind(u) = 0$.
	\end{enumerate}
\end{prop}
\begin{proof}
	$(2) \implies (1)$ is clear. 
	
	$(1) \implies (2)$ is a consequence of \cite{Gui12}. 
	
	It remains to show $(2) \iff (3) \iff (4) \iff (5)$. The first three of these equivalences follow from Theorems \ref{theo:degiorgi.conjecture.n2}, \ref{theo:degiorgi.conjecture.n3}. Next, $(4) \implies (5)$ is trivial by the definition of stability. Finally, $(5) \implies (3)$ follows because $\langle \nabla u, \mathbf{e} \rangle$ is an $L^\infty$ Jacobi field, so, by Theorem \ref{theo:app.pde.courant.nodal.domain}, it is either identically zero or has constant sign. Since $u$ is not constant, there will exist at least one $\mathbf{e} \in \SS^1$ for which $\langle \nabla u, \mathbf{e} \rangle > 0$. For an alternative proof of $(5) \implies (2)$, see the work of Farina-Mari-Valdinoci \cite{FarinaMariValdinoci13}.
\end{proof}

\section{Local results} \label{sec:local.results}

\subsection{General critical points}

In this section we collect preliminary results about general critical points that we will to invoke throughout the remainder of the paper. Key to transferring the local results of Hutchinson and Tonegawa to the manifold setting is an almost-monotonicity inequality from \cite[Appendix B]{Guaraco15}:
\begin{equation} \label{eq:monotonicity.identity} \frac{d}{dr} \Big[ e^{mr} r^{1-n} (E_\varepsilon \restr B_{r}(p))[u] \Big] \geq e^{m\rho} r^{1-n} \int_{B_r(p)} (-\xi_\varepsilon[u]) \, d\upsilon	
\end{equation}
for $r \in (0, \iota_\Sigma)$, $m = m(\sup_U |\sect_{\Sigma,g}|) > 0$. One then derives Lemmas \ref{lemm:xi.upper.bound} and \ref{lemm:upper.density.bounds} from it in the same way as in \cite{HutchinsonTonegawa00}.

\begin{lemm}[Hutchinson-Tonegawa {\cite[Proposition 3.3]{HutchinsonTonegawa00}}, cf. Guaraco {\cite[Appendix B]{Guaraco15}}] \label{lemm:xi.upper.bound}
	Let $(\Sigma^n, g)$ be a complete Riemannian manifold, $U \subset \Sigma \setminus \partial \Sigma$ be open, $\varepsilon > 0$, and $u$ be a critical point of $E_\varepsilon \restr U$ with $|u| \leq 1$. If $\varepsilon \leq \varepsilon_0$, then
	\[ \sup_{U'} \xi_\varepsilon[u] \leq c_0 \]
	for all $U' \subset \subset U$, where
	\begin{align*} 
		c_0 & = c_0(\sup_U |\sect_{\Sigma,g}|, \inf_U \inj_{\Sigma,g},  \dist_g(U', \partial U), W), \\
		\varepsilon_0 & = \varepsilon_0(\sup_U |\sect_{\Sigma,g}|, \inf_U \inj_{\Sigma,g}, \dist_g(U', \partial U), W).
	\end{align*}
\end{lemm}

\begin{coro} \label{coro:grad.u.harnack.inequality}
	Let $(\Sigma^n, g)$ be a complete Riemannian manifold, $U \subset \Sigma \setminus \partial \Sigma$ be open, $\varepsilon > 0$, and $u$ be a critical point of $E_\varepsilon \restr U$ with $|u| \leq 1$. If $\varepsilon \leq \varepsilon_1$, then
	\[ \varepsilon |\nabla \log (u^2 -1)| \leq c_1 \text{ on } U' \setminus \{ \nabla u = 0 \}, \]
	for all $U' \subset \subset U$, where
	\begin{align*} 
		c_1 & = c_1(\sup_U |\sect_{\Sigma,g}|, \inf_U \inj_{\Sigma,g}, \dist_g(U', \partial U), W), \\
		\varepsilon_1 & = \varepsilon_1(\sup_U |\sect_{\Sigma,g}|, \inf_U \inj_{\Sigma,g}, \dist_g(U', \partial U), W).
	\end{align*}
\end{coro}
\begin{proof}
	Without loss of generality, assume $u \geq 0$. By Lemma \ref{lemm:xi.upper.bound},
	\begin{align*} 
		& 2 c_0 \varepsilon \geq 2 \varepsilon \xi_\varepsilon[u] = \varepsilon^2 |\nabla u|^2 - 2W(u) = \frac{\varepsilon^2 |\nabla u|^2}{(u-1)^2} - \frac{2W(u)}{(u-1)^2} \\
		& \qquad \implies \frac{\varepsilon^2 |\nabla u|^2}{(u-1)^2} \leq 2 c_0 \varepsilon + \frac{2W(u)}{(u-1)^2}.
	\end{align*}
	This alone is sufficient to obtain the result when $0 \leq u \leq 1-\alpha$, so let's suppose $1 - \alpha < u \leq 1$. The result follows from Taylor's theorem, \eqref{assu:double.well.potential.nonnegative}, and \eqref{assu:double.well.potential.convex}.
\end{proof}

\begin{lemm}[Hutchinson-Tonegawa {\cite[Proposition 3.4]{HutchinsonTonegawa00}}, cf. Guaraco {\cite[Appendix B]{Guaraco15}}] \label{lemm:upper.density.bounds}
    Let $(\Sigma^n, g)$ be a complete Riemannian manifold, $U \subset \Sigma \setminus \partial \Sigma$ be open, $\varepsilon > 0$, and $u$ be a critical point of $E_\varepsilon\restr U$ with $|u| \leq 1$ and $(E_\varepsilon \restr U)[u] \leq E_0$. If $\varepsilon \leq \varepsilon_2$, then 
    \[ 2 \Vert V_\varepsilon[u] \Vert(B_r(p)) \leq (E_\varepsilon \restr B_r(p))[u] \leq c_2 r^{n-1} \text{ for } r \in (0, r_2] \]
	for all $U' \subset \subset U$, $p \in U'$, $r < r_2$, where
	\begin{align*} 
		c_2 & = c_2(n, \sup_U |\sect_{\Sigma,g}|, \inf_U \inj_{\Sigma,g}, \dist_g(U', \partial U), W, E_0), \\
		\varepsilon_2 & = \varepsilon_2(\sup_U |\sect_{\Sigma,g}|, \inf_U \inj_{\Sigma,g}, \dist_g(U', \partial U), W, E_0), \\
		r_2 & = r_2(\sup_U |\sect_{\Sigma,g}|, \inf_U \inj_{\Sigma,g}, \dist_g(U', \partial U)).
	\end{align*}
\end{lemm}

The lemma below due to Tonegawa goes through verbatim as in the Euclidean case:

\begin{lemm}[Tonegawa {\cite[Lemma 2.5]{TonegawaWickramasekera12}}] \label{lemm:grad.xi.pointwise.estimate}
    Let $(\Sigma^n, g)$ be a Riemannian manifold, $U \subset \Sigma \setminus \partial \Sigma$ be open, $\varepsilon > 0$, and $u$ be a critical point of $E_\varepsilon \restr U$. Then
    \[ |\nabla \xi_\varepsilon[u]| \leq \sqrt{n-1} \cdot \varepsilon |\nabla u| \sqrt{|\nabla^2 u|^2 - |\nabla |\nabla u||^2} \text{ on } U \cap \{ \nabla u \neq \mathbf{0} \}. \]
    Therefore,
    \[ |\nabla \xi_\varepsilon[u]| \leq \sqrt{n-1} \cdot \varepsilon |\nabla u|^2 |\cA| \text{ on } U \cap \{ \nabla u \neq \mathbf{0} \}. \]
\end{lemm}

\begin{lemm} \label{lemm:epsilon.curvature.bound}
	Let $(\Sigma^n, g)$ be a complete Riemannian manifold, $U \subset \Sigma \setminus \partial \Sigma$ be open and convex, $\varepsilon > 0$, and $u$ be a critical point of $E_\varepsilon \restr U$ with $|u| \leq 1$. If $p \in U' \subset \subset U$, $|u(p)| \leq 1-\gamma$, $\varepsilon |\nabla u| \geq \mu$ on $U \cap \{ |u| \leq 1-\beta \}$, $\beta < \gamma$, and $\varepsilon \leq \varepsilon_5$, then
	\[ \varepsilon |\cA| + \varepsilon^2 |\nabla \cA| \leq c_5 \text{ on } U' \cap B_{\theta_5 \varepsilon}(p). \]
	where
	\begin{align*}
		c_5 & = c_5(n, \sup_U |\sect_{\Sigma,g}|, \inf_U \inj_{\Sigma,g}, \dist_g(U', \partial U), \beta, \gamma, \mu, W), \\
		\varepsilon_5 & = \varepsilon_5(\sup_U |\sect_{\Sigma,g}|, \inf_U \inj_{\Sigma,g}, \dist_g(U', \partial U), W), \\
		\theta_5 & = \theta_5(n, \sup_U |\sect_{\Sigma,g}|, \inf_U \inj_{\Sigma,g}, \dist_g(U', \partial U), \beta, \gamma, \mu, W). \\
	\end{align*}
\end{lemm}
\begin{proof}
	Suppose $q \in U$. Then, from Lemma \ref{lemm:xi.upper.bound} we find that
	\[ \log (1-u(p)^2) - \log (1-u^2) \leq c_0 \varepsilon^{-1} r \]
	on $B_r(p) \subset U$, so $|u| \leq 1-\beta$ as long as $\varepsilon^{-1} r \leq c(\beta, \gamma, c_0)$. By assumption, then, $\varepsilon |\nabla u| \geq \mu$ on all these points, and by a direct calculation, together with elliptic regularity, we find that
	\[ \varepsilon |\cA| + \varepsilon^2 |\nabla \cA| \leq c \text{ on } B_{\theta \varepsilon}(p), \]
	as claimed.
\end{proof}

\subsection{Stable critical points} \label{sec:index.0.local.results}

We will view stability in the language of the enhanced second fundamental form. This make the proofs reminiscent of the corresponding stable minimal hypersurface theory.

\begin{lemm}[cf. Tonegawa {\cite[Proposition 1]{Tonegawa05}}] \label{lemm:stability.inequality}
    Let $(\Sigma^n, g)$ be a complete Riemannian manifold, $U \subset \Sigma \setminus \partial \Sigma$ be open, $\varepsilon > 0$, and $u$ be a stable critical point of $E_\varepsilon \restr U$. For every $\zeta \in C^\infty_c(U)$,
    \begin{align*} 
        & \int_{U \cap \{\nabla u = \mathbf{0}\}} |\nabla^2 u|^2 \zeta^2 \, d\upsilon \\
        & \qquad + \int_U (|\cA|^2 + \ricc_{\Sigma,g}(\nu, \nu)) \zeta^2 \, d\Vert V_\varepsilon[u] \Vert \leq \int_U |\nabla \zeta|^2 \, d\Vert V_\varepsilon[u] \Vert.
    \end{align*}
\end{lemm}

Here $\ricc_{\Sigma,g}$ denotes the Ricci curvature tensor of $(\Sigma^n, g)$, and $\nabla$ denotes the full ambient covariant derivative (\emph{not} the tangential derivative on the level sets). Notice that, even though $\cA$ and $\nu$ only make sense on $\Sigma \setminus \{ \nabla u = \mathbf{0} \}$, their integral can be taken over all of $U$ because
\[ \Vert V_\varepsilon[u] \Vert(\{ \nabla u = \mathbf{0} \} \cap U) = 0. \]

\begin{proof}
    If $u$ is a stable critical point for $E_\varepsilon \restr U$, then
    \[ \int_\Sigma |\nabla \psi|^2 - W''(u) \psi^2 \, d\upsilon \geq 0 \]
    for all $\psi \in C^\infty_c(U)$. The result will follow by plugging in
    \[ \psi = \zeta (|\nabla u|^2 + \delta)^{1/2}, \; \zeta \in C^\infty_c(U), \; \delta > 0, \]
    using the Bochner formula,
    \[ \frac{1}{2} \Delta |\nabla u|^2 = |\nabla^2 u|^2 + \langle \nabla u, \nabla \Delta u \rangle + \ricc(\nabla u, \nabla u), \]
    and finally sending $\delta \downarrow 0$.
\end{proof}

\begin{lemm} \label{lemm:sff.local.l2.estimate}
	Let $(\Sigma^2, g)$ be a complete Riemannian manifold, $U \subset \Sigma \setminus \partial \Sigma$ be open and bounded, and $u$ be a stable critical point of $E_\varepsilon \restr U$ with $|u| \leq 1$ and $(E_\varepsilon \restr U)[u] \leq E_0$. If $\varepsilon \leq \varepsilon_3$, then
	\[ \int_{B_r(p)} |\cA|^2 \, d\Vert V \Vert \leq c_3 \left( \frac{\dist_g(p, \partial U)}{(\dist(p, \partial U) - r)^2} + r \sup_U |\ricc_{\Sigma,g}| \right) \]
	for all $p \in U' \subset \subset U$, $r < \dist_g(p, \partial U)$, where
	\begin{align*}
		c_3 & = c_3(\sup_U |\sect_{\Sigma,g}|, \inf_U \inj_{\Sigma,g}, \dist_g(U', \partial U), W, E_0), \\
		\varepsilon_3 & = \varepsilon_3(\sup_U |\sect_{\Sigma,g}|, \inf_U \inj_{\Sigma,g}, \dist_g(U', \partial U), W, E_0).
	\end{align*}
\end{lemm}

\begin{proof}
	Construct a cutoff function $\zeta : \Sigma \to \RR$ such that $\zeta = 1$ on $B_r(p)$, $\zeta = 0$ off $U$, and $|\nabla \zeta| \leq c (\dist_g(p, \partial U)-r)^{-1}$, where $c$ depends on the local geometry around $p$. Then
	\begin{align*}
		& \int_{B_r(p)} |\cA|^2 \, d\Vert V \Vert \leq \int_U |\cA|^2 \zeta^2 \, d\Vert V \Vert \leq \int_U |\nabla \zeta|^2 - \ricc_{\Sigma,g}(\nu, \nu) \zeta^2 \, d\Vert V \Vert \\
		& \qquad \leq c \left( \frac{\dist_g(p, \partial U)}{(\dist_g(p, \partial U) -r)^2} + r \sup_U |\ricc_{\Sigma,g}| \right),
	\end{align*}
	by virtue of the upper density estimate in Lemma \ref{lemm:upper.density.bounds}.
\end{proof}

This gives:

\begin{coro} \label{coro:grad.xi.localized.stability}
	Let $(\Sigma^2, g)$ be a complete Riemannian manifold, $U \subset \Sigma \setminus \partial \Sigma$ be open and bounded, and $u$ be a stable critical point of $E_\varepsilon \restr U$ with $|u| \leq 1$ and $(E_\varepsilon \restr U)[u] \leq E_0$. If $\varepsilon \leq \varepsilon_4$, then
	\[ \int_{B_r(p)} |\nabla \xi_\varepsilon[u]| \, d\upsilon \leq c_4  \left( \frac{r \cdot \dist_g(p, \partial U)}{(\dist_g(p, \partial U) - r)^2} + r \sup_U |\ricc_{\Sigma,g}| \right)^{\frac{1}{2}} \]
	for every $p \in U' \subset \subset U$, $r < \dist_g(p, \partial U)$, where
	\begin{align*}
		c_4 & = c_4(\sup_U |\sect_\Sigma|, \inf_U \inj_{\Sigma,g}, \dist_g(U', \partial U), W, E_0), \\
		\varepsilon_4 & = \varepsilon_4(\sup_U |\sect_{\Sigma,g}|, \inf_U \inj_{\Sigma,g}, \dist_g(U', \partial U), W, E_0).
	\end{align*}
\end{coro}
\begin{proof}
    Combining Lemma \ref{lemm:grad.xi.pointwise.estimate} with the H\"older inequality and the upper density bound, we have
    \begin{align*}
        & \int_{B_r(p)} |\nabla \xi_\varepsilon[u]| \, d\upsilon \leq \sqrt{n-1} \int_{B_r(p)} |\cA| \, d\Vert V \Vert \\
        & \qquad \leq \sqrt{n-1} \left( \int_{B_r(p)} |\cA|^2 \, d\Vert V \Vert \right)^{\frac{1}{2}} \Vert V \Vert(B_r(p))^{\frac{1}{2}} \\
        & \qquad \leq c \left( \frac{\dist_g(p, \partial U)}{(\dist_g(p, \partial U) - r)^2} + r \sup_U |\ricc_{\Sigma,g}| \right)^{\frac{1}{2}} \Vert V \Vert(B_r(p))^{\frac{1}{2}} \\
        & \qquad \leq c \left( \frac{r \cdot \dist_g(p, \partial U)}{(\dist_g(p, \partial U) - r)^2} + r^2 \sup_U |\ricc_{\Sigma,g}| \right)^{\frac{1}{2}}.
    \end{align*}
    This is the required result.
\end{proof}

\begin{rema}
	Both Lemma \ref{lemm:stability.inequality} and Corollary \ref{coro:grad.xi.localized.stability} can be sharpened by replacing $\sup_U |\ricc_{\Sigma,g}|$ by
	\[ \max \left\{ 0, \sup_{U \cap \{ \nabla u \neq \mathbf{0} \}} \ricc_{\Sigma,g}(\nu, \nu) \right\}. \]
\end{rema}

The flat version of the $L^1$ gradient estimate in \ref{coro:grad.xi.localized.stability} was a key fact in the proof of Theorem \ref{theo:index.0.main} below in \cite{Tonegawa05} in the flat two-dimensional setting. Specifically, by the Neumann-Poincar\'e inequality we find that
\[ \int_{B_r(p)} |\xi_\varepsilon[u] - \overline{(\xi_\varepsilon[u])}_{p,r}|^2 \, d\upsilon \leq c \left( \frac{r \cdot \dist_g(p, \partial U)}{(\dist_g(p, \partial U) - r)^2} + r^2 \sup_U |\ricc_{\Sigma,g}| \right) \]
and, therefore,

\begin{coro} \label{coro:xi.local.l2.estimate}
	Let $(\Sigma^2, g)$ be a complete Riemannian manifold, $U \subset \Sigma \setminus \partial \Sigma$ be open and bounded, and $u$ be a stable critical point of $E_\varepsilon \restr U$ with $|u| \leq 1$ and $(E_{\varepsilon} \restr U)[u] \leq E_0$. If $\varepsilon \leq \varepsilon_5$, then
	\begin{align*} 
		\int_{B_r(p)} |\xi_\varepsilon[u]|^2 \, d\upsilon & \leq c_5 \left( \frac{r \cdot \dist_g(p, \partial U)}{(\dist_g(p, \partial U) - r)^2} + r^2 \sup_U  |\ricc_{\Sigma,g}| \right) \\
			& \qquad + \frac{2}{\upsilon(B_r(p))} \left( \int_{B_r(p)} \xi_\varepsilon[u] \, d\upsilon \right)^2
	\end{align*}
	for every $p \in U' \subset \subset U$, $r < \dist_g(p, \partial U)$, where
	\begin{align*}
		c_5 & = c_5(\sup_U |\sect_\Sigma|, \inf_U \inj_{\Sigma,g}, \dist_g(U', \partial U), W, E_0), \\
		\varepsilon_5 & = \varepsilon_5(\sup_U |\sect\Sigma|, \inf_U \inj_{\Sigma,g}, \dist_g(U', \partial U), W, E_0).
	\end{align*}
\end{coro}

\begin{lemm} \label{lemm:index.0.lower.density}
    Let $(\Sigma^2, g)$ be a complete Riemannian manifold, $U \subset \Sigma \setminus \partial \Sigma$ be open and bounded, and $u$ be a stable critical point of $E_\varepsilon \restr U$ with $|u| < 1$ and $(E_{\varepsilon} \restr U)[u] \leq E_0$. Let $\beta \in (0, 1)$. If $\varepsilon \leq \varepsilon_6$, then
    \[ \varepsilon |\nabla u(p)| \geq \frac{1}{4} \min_{|s| \leq 1-\beta/2} W(s), \]
    for every $p \in U' \cap \{ |u| \leq 1-\beta \}$, where $U' \subset \subset U$, and
    \begin{align*}
        \varepsilon_6 & = \varepsilon_6(\sup_U |\sect_{\Sigma,g}|, \inf_U \inj_{\Sigma,g}, \dist_g(U', \partial U), W, E_0, \beta).
    \end{align*}
\end{lemm}
\begin{proof}
    If the statement were false, there would exist a sequence of stable critical points $\{(u_i, \varepsilon_i) \}_{i=1,2,\ldots}$ of $E_{\varepsilon_i} \restr U_i$, $\varepsilon_i \to 0$, with $U_i' \subset \subset U_i \subset \Sigma_i$, sectional curvature and injectivity radius bounds, a fixed distance $\dist_{g_i}(U_i', \partial U_i)$, fixed energy bounds, and with points $p_i \in \{ |u_i| \leq 1-\beta \}$ such that 
    \[ \varepsilon_i |\nabla u_i(p_i)| < \omega \triangleq \frac{1}{4} \min_{|s| \leq 1-\beta/2} W(s). \]
    By elliptic regularity, there exists $\theta$ depending on the local geometry such that
	\[ \varepsilon_i \xi_{\varepsilon_i}[u_i] = \frac{\varepsilon_i^2}{2} |\nabla u_i|^2 - W(u_i) \leq  - \frac{7}{8} \omega^2 \text{ on } B_{\theta \varepsilon_i}(q) \]
	and, therefore,
	\begin{equation} \label{eq:index.0.density.i} \int_{B_{\theta \varepsilon_i}(p_i)} |\xi_{\varepsilon_i}[u_i]|^2 \, d\upsilon_{g_i} \geq c_0.
	\end{equation}
	Consider an intermediate scale $\lambda_i$ such that $\varepsilon_i \ll \lambda_i \ll \dist_{g_i}(p_i, \partial U_i)$. Restrict to $i$ large enough that
	\[ c \left( \frac{\lambda_i \cdot \dist_{g_i}(p_i, \partial U_i)}{(\dist_{g_i}(p_i, \partial U_i) - \lambda_i)^2} + \lambda_i^2 \sup_{U_i} |\ricc_{\Sigma_i,g_i}| \right) \leq c_0/2, \]
	where $c_0$ is as in \eqref{eq:index.0.density.i} and $c$ is as in Corollary \ref{coro:xi.local.l2.estimate}. Invoking the corollary on $B_{\lambda_i}(p_i)$ and with $U = U_i$, we get
	\begin{equation} \label{eq:index.0.density.ii}
		\int_{B_{\lambda_i}(p_i)} |\xi_{\varepsilon_i}[u_i]|^2 \, d\upsilon_{g_i} \leq \frac{c_0}{2} + \frac{2}{\upsilon_{g_i}(B_{\lambda_i}(p_i))} \left( \int_{B_{\lambda_i}(p_i)} \xi_{\varepsilon_i}[u_i] \, d\upsilon_{g_i} \right)^2.
	\end{equation}
	Note that
	\[ \frac{1}{\lambda_i} \int_{B_{\lambda_i}(p_i)} \xi_{\varepsilon_i}[u_i] \, d\upsilon_{g_i} = \int_{B_1(0)} \xi_{\lambda_i^{-1} \varepsilon_i}[\widetilde{u}_i] \, d\upsilon_{\widetilde{g}_i}, \]
	where $\widetilde{u}_i(x) \triangleq u_i(p_i + \lambda_i x)$ and $\upsilon_{\widetilde{g}_i}$ is the volume form for the rescaled manifold $\lambda_i^{-1}(\Sigma - p_i)$. Note that each $\widetilde{u}_i$ is a critical point for $E_{\lambda_i^{-1} \varepsilon_i} \restr B_1(0)$, and that, by Lemma \ref{lemm:upper.density.bounds},
	\[ \sup_i (E_{\lambda_i^{-1} \varepsilon_i} \restr B_1(0))[\widetilde{u}_i] < \infty. \]
	Since $\lambda_i^{-1} \varepsilon_i \to 0$, Theorem \ref{theo:main.known} shows that
	\[ \lim_i \int_{B_1(0)} \xi_{\lambda_i^{-1} \varepsilon_i}[\widetilde{u}_i] \, d\upsilon_{\widetilde{g}_i} = 0. \]
	Plugging this into \eqref{eq:index.0.density.ii}, and recalling $\upsilon(B_{\lambda_i^{-1} \varepsilon_i}(0)) = \omega_2 \lambda_i^{-2} \varepsilon_i^2 + o(1)$:
	\[ \int_{B_{\lambda_i}(p_i)} |\xi_{\varepsilon_i}[u_i]|^2 \, d\upsilon \leq \frac{c_0}{2} + o(1). \]
	This contradicts \eqref{eq:index.0.density.i}, since $\lambda_i \geq \theta \varepsilon_i$ for sufficiently large $i$.
\end{proof}

With all these results at our disposal, we easily recover the following result originally due to Tonegawa in the flat two-dimensional setting:

\begin{theo}[cf. Tonegawa {\cite[Theorem 5]{Tonegawa05}}] \label{theo:index.0.main}
    Assume the same hypotheses as Theorem \ref{theo:main.known} and, additionally, that $\dim \Sigma = 2$ and that every $u_i$ is a stable critical point for $E_{\varepsilon_i} \restr (U, g_i)$. Then, all conclusions of Theorem \ref{theo:main.known} hold true, as well as $\sing \support \Vert V^\infty \Vert \cap U = \emptyset$. Moreover, for every $U' \subset \subset U$, $\beta \in (0, 1)$, $\theta \in (0, 1)$, there exists a $c > 0$ such that
    \[ \sup_{t \in I} [\nu^i]_{C^{1/2}(\{u_i = t \} \cap U')} \leq c, \]
    where $I \subset [-1+\beta,1-\beta]$ is measurable, with Lebesgue measure $\geq 2\theta(1-\beta)$, and $\nu^i$ denotes the unit normal vector to the level set curve of $u$ through each particular point.
\end{theo}

\subsection{Wang-Wei curvature estimates} \label{sec:wang.wei.curvature.estimates}

K. Wang and J. Wei have obtained a considerable strengthening of Theorem \ref{theo:index.0.main} in the flat setting. We refer to Appendix \ref{sec:app.wang.wei.manifolds} for the adaptation of their result to the curved setting, yielding:

\begin{theo}[cf. Wang-Wei {\cite[Theorem 3.7]{WangWei17}}] \label{theo:wang.wei.epsilon.curvature}
    Let $(\Sigma^2, g)$ be a complete Riemannian manifold, $U \subset \Sigma \setminus \partial \Sigma$ be open and bounded, and $u$ be a stable critical point of $E_\varepsilon \restr U$ with $|u| < 1$. If $\varepsilon \leq \varepsilon_*$ and
    \[ |\cA| \leq C \text{ on } U \cap \{|u| \leq 1-\beta\} \]
    then
    \[ |\cA| \leq c_* \varepsilon^{1/7} \text{ on } U' \cap \{|u| \leq 1-\beta\}, \]
    for all $U' \subset \subset U$, where
    \begin{align*}
        \varepsilon_* & = \varepsilon_*(\sup_U |\sect_{\Sigma,g}|, \inf_U \inj_{\Sigma,g}, \dist_g(U', \partial U), \beta, C, W), \\
        c_* & = c_*(\sup_U |\sect_{\Sigma,g}|, \inf_U \inj_{\Sigma,g}, \dist_g(U', \partial U), \beta, C, W).
    \end{align*}
\end{theo}

As a direct corollary of Theorem \ref{theo:wang.wei.epsilon.curvature} we have:

\begin{coro} \label{coro:index.0.bounded.curvature}
    Let $(\Sigma^2, g)$ be a complete Riemannian manifold, $U \subset \Sigma \setminus \partial \Sigma$ be open and bounded, and $u$ be a stable critical point of $E_\varepsilon \restr U$ with $|u| < 1$ and $(E_{\varepsilon} \restr U)[u] \leq E_0$. Let $\beta \in (0, 1)$. If $\varepsilon \leq \varepsilon_7$, then
    \[ |\cA| \leq c_7 \text{ on } U' \cap \{|u| \leq 1-\beta\}, \]
    for all $U' \subset \subset U$, where
    \begin{align*}
        c_7 & = c_7(\sup_U |\sect_{\Sigma,g}|, \inf_U \inj_{\Sigma,g}, \dist_g(U', \partial U), W, E_0, \beta), \\
        \varepsilon_7 & = \varepsilon_7(\sup_U |\sect_{\Sigma,g}|, \inf_U \inj_{\Sigma,g}, \dist_g(U', \partial U), W, E_0, \beta).
    \end{align*}
\end{coro}
\begin{proof}
    If the statement were false, there would exist a sequence of stable critical points $\{(u_i, \varepsilon_i) \}_{i=1,2,\ldots}$ of $E_{\varepsilon_i} \restr U_i$, $\lim_i \varepsilon_i = 0$, with $U_i' \subset \subset U_i \subset \Sigma_i \setminus \partial \Sigma_i$, sectional curvature and injectivity radius bounds, a fixed distance $\dist_{g_i}(U_i', \partial U_i)$, fixed energy bounds,  and such that 
    \[ \sup_{U_i \cap \{|u_i| \leq 1-\beta\}} |\cA_i| \dist_{g_i}(\cdot, \partial U_i) \]
    is unbounded as $i \uparrow \infty$. Denote by $p_i$ the point at which the supremum is attained, $\lambda_i^{-1} \triangleq |\cA_i(p_i)|$, and $\delta_i \triangleq \dist_{g_i}(p_i, \partial U_i)$. Combining Lemma \ref{lemm:index.0.lower.density} with Lemma \ref{lemm:epsilon.curvature.bound}, we find that
    \[ \limsup_{i \uparrow \infty} \lambda_i^{-1} \varepsilon_i < \infty. \]
    \begin{clai}
        $\lim_{i \uparrow \infty} \lambda_i^{-1} \varepsilon_i = 0$.
    \end{clai}
    \begin{proof}
        Proceed by contradiction. If the $\limsup$ were a positive real number, then, after passing to a subsequence, the $\lambda_i$-blowups would have uniformly elliptic estimates and would subsequentially converge in $C^\infty_{\loc}(\RR^2)$ to an entire, stable solution $\widetilde{u}_\infty : \RR^2 \to (-1,1)$ with $|\widetilde{\cA}_\infty(0)| = 1$; this contradicts the known fact that the only such $\widetilde{u}_\infty$ is the lift of the one-dimensional heteroclinic solution to $\RR^2$ (see, e.g., Proposition \ref{prop:m2.characterization}).
    \end{proof}
    
    Thus,
    \begin{equation} \label{eq:index.0.bounded.curvature.i}
        \lim_{i \uparrow \infty} \lambda_i^{-1} \varepsilon_i = 0.
    \end{equation}
    Rescaling to
    \[ \widetilde{u}_i(x) \triangleq u_i(p_i + \lambda_i x), \; x \in \widetilde{\Sigma}_i \triangleq \lambda_i^{-1}(\Sigma_i - p_i), \]
    we land in the setting of Theorem \ref{theo:index.0.main}. This violates Theorem \ref{theo:wang.wei.epsilon.curvature}, seeing as to how $|\widetilde{\cA}_i(0)| = 1 \leq c (\lambda_i^{-1} \varepsilon_i)^{1/7}$ is false for large $i$.
\end{proof}

\subsection{Index 1 singularity formation} \label{sec:index.1.singularity.formation}

In this section we work toward understanding the limiting picture of a sequence of Morse index 1 critical points $\{ (u_i, \varepsilon_i) \}_{i=1,2,\ldots}$ with uniform energy bounds and $\lim_i \varepsilon_i = 0$. We make the following definition:

\begin{defi} \label{def:index.1.localizations}
    For every critical point $u$ of $E_\varepsilon \restr U$ define the following collection of open subsets of $U$:
    \[ \cI_{1,\varepsilon}[u] \triangleq \{ \text{open subsets } U' \subset U \text{ with } \ind(u; U') = 1 \}. \]
\end{defi}

The following is a trivial consequence for Morse index-1 critical points:

\begin{lemm} \label{lemm:index.localization.structure}
    For every Morse index-1 critical point $u$ of $E_\varepsilon \restr U$,
    \begin{enumerate}
        \item $U \in \cI_{1,\varepsilon}[u]$,
        \item $V_1, V_2 \in \cI_{1,\varepsilon}[u] \implies V_1 \cap V_2 \neq \emptyset$,
        \item $V \in \cI_{1,\varepsilon}[u] \implies V \cap \{ |u| \leq 1-\alpha \} \neq \emptyset$,
    \end{enumerate}
    where $\alpha$ is as in \eqref{assu:double.well.potential.convex}
\end{lemm}

The following preliminary understanding of the limiting picture is a byproduct of Lemma \ref{lemm:index.localization.structure}:

\begin{lemm} \label{lemm:index.1.crude.picture}
    Assume the same hypotheses as Theorem \ref{theo:main.known} and, additionally, that $\dim \Sigma = 2$ and that every $u_i$ is a critical point for $E_{\varepsilon_i} \restr (U, g_i)$ with $\ind(u_i; (U, g_i)) \leq 1$. Then, all conclusions of Theorem \ref{theo:main.known} hold true, and $\cH^0(\sing \support \Vert V^\infty \Vert \cap U) \leq 1$.
\end{lemm}
\begin{proof}
	Suppose, for the sake of contradiction, that $\cH^0(\sing \Vert V^\infty \Vert) \geq 2$. Then there would exist two disjoint open subsets $V_1$, $V_2$ of $U$ with $\sing V^\infty \cap V_j \neq \emptyset$ for $j = 1$, $2$. By Lemma \ref{lemm:index.localization.structure}, $V_1$, $V_2$ cannot simultaneously be in $\cI_{1,\varepsilon_i}[u_i]$ for any $i = 1, 2, \ldots$ Passing to a subsequence, we can arrange, e.g., that $u_i$ is a stable critical point of $E_{\varepsilon_i} \restr V_1$. But then Theorem \ref{theo:index.0.main} would force $\sing \Vert V^\infty \Vert \cap V_1 = \emptyset$, a contradiction.
\end{proof}

\begin{prop} \label{prop:index.1.main}
    Let $(\Sigma^2, g)$ be a closed Riemannian 2-manifold. Let $\{ (u_i, \varepsilon_i) \}_{i=1,2,\ldots} \subset C^\infty(\Sigma) \times (0,\infty)$, $\lim_i \varepsilon_i = 0$, $|u_i| < 1$, $E_{\varepsilon_i} [u_i] \leq E_0$, and where each $u_i$ is a nonconstant critical point of $E_{\varepsilon_i}$ with $\ind(u_i) = 1$. If $\sing \support \Vert V^\infty \Vert \neq  \emptyset$, then
    \[ \Theta^1(\energyunit^{-1} \Vert V^\infty \Vert, p_*) = 2 \]
    at the (unique) $p_* \in \sing \support \Vert V^\infty \Vert$.
\end{prop}
\begin{proof}
    The only case not covered by Lemma \ref{lemm:index.1.crude.picture} above is that in which $\sing \support \Vert V^\infty \Vert = \{ p_* \}$ with $\Theta^1(\energyunit^{-1} \Vert V^\infty \Vert, p_*) \geq 3$. We will deal with that case here. Fix $\beta \in (0, 1)$, and let $\omega$ be as in Lemma \ref{lemm:index.0.lower.density}. 
    
    By Lemma \ref{lemm:index.0.lower.density} and Lemma \ref{lemm:index.localization.structure}, exactly one of the following is true:
    \begin{enumerate}
        \item there exists $D > 0$ such that, after discarding at most finitely many $\{(u_i, \varepsilon_i)\}_{i=1,2,\ldots}$, we have for each $i = 1, 2, \ldots$ a nonempty open set $\cS_i \subset \{ |u_i| \leq 1-\beta \}$ with
            \begin{multline} \label{eq:index.1.main.i}
                \varepsilon_i |\nabla u_i| \geq \omega \text{ on } \{ |u_i| \leq 1-\beta \} \setminus \cS_i, \text{ and } \\
                q \in \cS_i \implies \cS_i \subset B_{2D\varepsilon_i}(q) \text{ and } B_{D\varepsilon_i}(q) \in \cI_{1,\varepsilon_i}[u_i],
            \end{multline}
        \item or, alternatively, we can pass to a subsequence along which
            \begin{equation} \label{eq:index.1.main.ii}
                \varepsilon_i |\nabla u_i| \geq \omega \text{ on } \{ |u_i| \leq 1-\beta \} \text{ for all } i = 1, 2, \ldots;
            \end{equation}
            in this case, set $\cS_i \triangleq \emptyset$ for all $i$ and $D = 0$.
    \end{enumerate}
    
    For $i = 1, 2, \ldots$, define
    \begin{multline*}
        \cR_{1,\varepsilon_i}[u_i] = \inf \Big\{ r \geq 2D\varepsilon_i : \text{ there exists } p \in \{ |u_i| \leq 1-\beta \} \\
        \text{ such that } B_r(p) \in \cI_{1,\varepsilon_i}[u_i] \Big\}.
    \end{multline*}
    Write $r_i = \cR_{1,\varepsilon_i}[u_i]$, and pick any $p_i$ such that $B_{2r_i}(p_i) \in \cI_{i,\varepsilon_i}[u_i]$.
    
    \begin{clai}
        $\lim_i r_i = 0$.
    \end{clai}
    \begin{proof}[Proof of claim]
        Suppose
        \[ \lim_i r_i = 2\sigma > 0. \]
        By Theorem \ref{theo:main.known}, $\{ u_i = 0 \}$ converges in the Hausdorff topology to $\support \Vert V^\infty \Vert$, so there would exist $q_i \in \{ |u_i| \leq 1-\beta \}$ with $\lim_i q_i = p_*$ and $\ind(u_i; B_\sigma(q_i)) = 0$. Thus, by Theorem \ref{theo:index.0.main}, it would follow that $\sing \support \Vert V^\infty \Vert \cap B_{\sigma}(p_*) = \emptyset$, which is clearly a contradiction.
    \end{proof}
    
    As a byproduct of the proof of this claim, it is easy to see that $\lim_i p_i = p_*$. Let $R > 0$ be a radius smaller than the injectivity radius of $\Sigma$ so that $\Sigma \cap B_R(q) \approx B_R$ for all $q \in \Sigma$, and define
    \[ \delta_i(x) \triangleq \dist(x; (\Sigma \setminus B_R(p_i)) \cup \closure{B}_{2r_i}(p_i)). \]
    By Lemma \ref{lemm:index.localization.structure}, $u_i$ is a stable critical point of $E_{\varepsilon_i} \restr (B_R(p_i) \setminus \closure{B}_{2r_i}(p_i))$ and, by \eqref{eq:index.1.main.i},
    \begin{equation} \label{eq:index.1.main.iii}
        \varepsilon_i |\nabla u_i| \geq \omega \text{ on } \{ |u_i| \leq 1-\beta  \} \setminus B_{2r_i}(p_i).
    \end{equation}
    
    \begin{clai}
        $|\cA_i| \delta_i \leq c$ on $\{ |u_i| \leq 1-\beta \}$, with $c$ independent of $i = 1, 2, \ldots$
    \end{clai}
    \begin{proof}[Proof of claim]
        Let $q_i \in \{ |u_i| \leq 1- \beta \} \cap B_{R}(p_i) \setminus \closure{B}_{2r_i}(p_i)$ denote a point attaining the maximum of $\{ |u_i| \leq 1-\beta \} \ni q \mapsto |\cA_i(q)| \delta_i(q)$, and suppose that the corresponding maximum values $\{|\cA_i(q_i)|\delta_i(q_i)\}_{i=1,2,\ldots}$ form an unbounded sequence. If $\lambda_i^{-1} \triangleq |\cA_i(q_i)|$, then consider the rescaled functions
        \[ \widetilde{u}_i(x) \triangleq u_i(q_i + \lambda_i x), \; x \in \widetilde{U}_i \triangleq \lambda_i^{-1}(B_{\delta_i(q_i)/2}(q_i) - q_i), \]
        with corresponding $\widetilde{\varepsilon}_i = \lambda_i^{-1} \varepsilon_i$. Arguing as in Corollary \ref{coro:index.0.bounded.curvature}, we can check that $\lim_i \widetilde{\varepsilon}_i = 0$. Then, $|\widetilde{\cA}_i| \leq 2$ on $\widetilde{U}_i \cap \{ |u_i| \leq 1-\beta \}$, so, by Theorem \ref{theo:wang.wei.epsilon.curvature} we get the improved estimate
        \[ |\widetilde{\cA}_i(0)| \leq c \widetilde{\varepsilon}_i^{1/7} \to 0, \]
        contradicting the normalization $|\widetilde{\cA}_i(0)| = 1$; the claim follows.
    \end{proof}
    
    Without loss of generality, suppose $2D > \varepsilon_*^{-1}$. Consider the open sets
    \[ W_i \triangleq \{ q \in B_{R/2}(p_i) : \delta_i(q) > r_i, \; |\delta_i(q)^{-1} (g - p_i) - \delta|_{C^3(B_1)} < \varepsilon_* \}. \]
    By virtue of the claim above, Theorem \ref{theo:wang.wei.epsilon.curvature} applies to give
    \begin{equation} \label{eq:index.1.main.iv}
        |\cA_i(q)| \leq c_* \varepsilon_i^{1/7} \delta_i(q)^{-8/7} \text{ for } q \in W_i \cap \{ |u_i| \leq 1-\beta \}.
    \end{equation}
    
    Let $\{ \sigma_i \}_{i=1,2,\ldots} \subset (0,\infty)$ be such that
    \begin{multline} \label{eq:index.1.main.v}
        \lim_i \sigma_i = \lim_i \sigma_i^{-1} r_i = 0, \text{ and } \\
        \lim_i \sigma_i^{-1} (V^i - p_i) \restr \grassmanian(B_1(0)) = (T_{p_*} V^\infty - p_*) \restr \grassmanian(B_1(0)).
    \end{multline}
    Denote $C = T_{p_*} V^\infty - p_*$, a singular cone with $\Theta^1(\energyunit^{-1} \Vert C \Vert, 0) \geq 3$. Moreover, denote
    \begin{multline*}
        \widetilde{u}_i(x) \triangleq u_i(p_i + \sigma_i x), \; x \in \widetilde{U}_i \triangleq \sigma_i^{-1}(B_{R/2}(p_i) - p_i), \\
        \widetilde{\varepsilon}_i = \sigma_i^{-1} \varepsilon_i, \; \widetilde{\delta}_i(\cdot) \triangleq \sigma_i^{-1} \delta_i(p_i + \sigma_i \cdot), \; \widetilde{r}_i \triangleq \sigma_i^{-1} r_i.
    \end{multline*}
    From \eqref{eq:index.1.main.iv} and Theorem \ref{theo:index.0.main} we know that, after perhaps passing to a subsequence,
    \begin{multline} \label{eq:index.1.main.vi}
    	\lim_i \{ \widetilde{u}_i = t_i \} \cap B_1(0) \setminus B_{1/2}(0) \\
    	= \support \Vert C \Vert \cap B_1(0) \setminus B_{1/2}(0) \text{ in } C^{1,\theta}
    \end{multline}
    for all $\theta < \frac{1}{2}$; here, $|t_i| \leq 1-\beta$. Consider a point $\widetilde{q}_i \in \{ \widetilde{u}_i = t_i \} \cap \partial B_{3/4}(0)$, and let $\widetilde{\nu}^\infty$ denote the unit normal vector (unique up to $\pm$) orthogonal to $\support \Vert C \Vert$ at $\lim_i \widetilde{q}_i$, so that $\lim_i \widetilde{\nu}^i(\widetilde{q}_i) = \widetilde{\nu}^\infty$ by \eqref{eq:index.1.main.vi}. Consider an arclength parametrization $\widetilde{\gamma}_i$ of $\{ \widetilde{u}_i = 0 \} \cap B_1(0)$, with $\widetilde{\gamma}_i(0) = \widetilde{q}_i$ and $\widetilde{\gamma}_i'(0)$ pointing toward the origin.
    
    Denote
    \[ T_i = \min \left\{ t \geq 0 : \widetilde{\gamma}_i \not \in W_i \text{ or } |\widetilde{\nu}^i(\widetilde{\gamma}_i(t)) - \widetilde{\nu}^\infty| \geq \frac{1}{4} \right\}. \]
    Note that \eqref{eq:index.1.main.iv} applies---rescaled---and gives
    \[ |\widetilde{\cA}_i(\widetilde{\gamma}_i(t))| \leq c \widetilde{\varepsilon}_i^{\frac{1}{7}} \widetilde{\delta}_i(\widetilde{\gamma}_i(t))^{-\frac{8}{7}} \leq c \widetilde{\varepsilon}_i^{\frac{1}{7}} (T_i + \widetilde{r}_i - t)^{-\frac{8}{7}}, \; t \in [0, T_i], \]
    and thus, from the fundamental theorem of calculus, for all $\tau \in [0, T_i]$,
    \begin{multline} \label{eq:index.1.main.vii}
        |\widetilde{\nu}^i(\widetilde{\gamma}_i(\tau)) - \widetilde{\nu}^i(\widetilde{q}_i)| \leq \int_0^{T_i} |\widetilde{\cA}_i(\widetilde{\gamma}_i(t))| \, dt \leq c \widetilde{\varepsilon_i}^{\frac{1}{7}} \int_0^{T_i} (T_i + \widetilde{r}_i - t)^{-\frac{8}{7}} \, dt \\
            \leq c \widetilde{\varepsilon}_i^{\frac{1}{7}} \left[ (T_i + \widetilde{r}_i - t)^{-\frac{1}{7}} \right]_{t=0}^{T_i} \leq c \widetilde{\varepsilon}_i^{\frac{1}{7}} \widetilde{r}_i^{-\frac{1}{7}} = c \varepsilon_i^{\frac{1}{7}} r_i^{-\frac{1}{7}}.
    \end{multline}
    
    \begin{clai}
    	$\liminf_i r_i^{-1} \varepsilon_i > 0$
    \end{clai}
    \begin{proof}[Proof of claim]
	    We now proceed to finish the proof of the claim. If the claim were false, then by passing to a subsequence we would be able to arrange that, along all rays of $\support \Vert C \Vert$, 
    	\[ |\widetilde{\nu}^i(\widetilde{\gamma}_i(t)) - \widetilde{\nu}^\infty| < \frac{1}{4} \text{ for all } t \in [0,T_i], \]
    	by virtue of \eqref{eq:index.1.main.v}. Thus, $\widetilde{\gamma}_i(T_i) \in \partial W_i$ from the definition of $T_i$, and
	    \begin{equation} \label{eq:index.1.main.viii}
	    	\widetilde{\delta}_i(\widetilde{\gamma}_i(T_i)) = \widetilde{r}_i \text{ and }  \lim_i \widetilde{\nu}^i(\widetilde{\gamma}_i(T_i)) = \widetilde{\nu}^\infty.
	    \end{equation}
	    By the definition of $\widetilde{r}_i$, the further blowup
	    \[ \widehat{u}_i(x) = \widetilde{u}_i(\widetilde{r}_i x), \; x \in \widehat{U}_i = \widetilde{r}_i^{-1} \widetilde{U}_i \]
	    is such that $B_2(0) \in \cI_{1,\widehat{\varepsilon}_i}[\widehat{u}_i]$, for $\widehat{\varepsilon}_i \triangleq \widetilde{r}_i^{-1} \widetilde{\varepsilon}_i = r_i^{-1} \varepsilon_i$, where, by assumption, $\lim_i \widehat{\varepsilon}_i = 0$. Moreover,
	    \begin{equation} \label{eq:index.1.main.ix}
	    	\cR_{1,\widehat{\varepsilon}_i}[\widehat{u}_i] = 1 \text{ for } i = 1, 2, \ldots
	    \end{equation}
		From \eqref{eq:index.1.main.viii} we see that
	    \[ \lim_i \widehat{V}^i = C, \]
	    which, in particular, has $\sing \support \Vert C \Vert = \{ 0 \}$. Therefore, arguing as in the first claim in the proof of Proposition \ref{prop:index.1.main}, we see that
	    \[ \lim_i \cR_{1,\widehat{\varepsilon}_i}[\widehat{u}_i] = 0, \]
	    contradicting \eqref{eq:index.1.main.ix}.
    \end{proof}
    
    Define
    \begin{equation} \label{eq:index.1.main.x}
    	\rho_i \triangleq \inf \left\{ \rho > 0 : \Theta^1(\Vert V^i \Vert, p_*, \rho) = \Theta^1(\Vert V^\infty \Vert, p_*) - \frac{\energyunit}{2} \right\}.
    \end{equation}
	(See Appendix \ref{sec:app.gmt} for the notation.) From Theorem \ref{theo:m2k.index.lower.bound}, Proposition \ref{prop:finite.index.equivalent.m2k}, and the uniformly elliptic estimates one gets in the $O(\varepsilon_i)$-scale, it follows that
	\[ \Theta^1(\Vert V^\infty \Vert, p_*) \geq 3 \implies \lim_i \rho_i^{-1} \varepsilon_i = 0. \]
	By the conclusion of the claim above, \eqref{eq:index.1.main.vii}, and the argument leading to \eqref{eq:index.1.main.viii}, it follows that the blowups
	\[ \breve{u}_i(x) \triangleq u(p_* + \rho_i x), \; x \in \rho_i^{-1}(B_{R/2}(p_*) - p_*) \]
	have $\breve{\varepsilon}_i \triangleq \rho_i^{-1} \varepsilon_i$ such that $\lim_i \breve{\varepsilon}_i = 0$ and the corresponding diffuse 1-varifolds $\breve{V}^i$ are such that
	\[ \lim_i \breve{V}^i = C, \]
	contradicting the non-integral density from \eqref{eq:index.1.main.x}.
\end{proof}

\section{Min-max construction} \label{sec:min.max}

\begin{proof}[Proof of Theorem \ref{theo:minmax.construction}]
	We break up the proof into the three steps outlined in the introduction.
	
	{\bf Step 1: Mountain pass.} (See \cite{Guaraco15} for details.) Denote
	\begin{multline*}
		\Gamma \triangleq \Big\{ \gamma \in C^0([-1,1]; W^{1,2}(\Sigma)) : \gamma(-1)(\cdot) \equiv -1, \text{ and} \\
		\gamma(1)(\cdot) \equiv 1 \text{ a.e. on } \Sigma \Big\}
	\end{multline*}
	Arguing as in \cite{Guaraco15}, one can show that the min-max energy levels
	\[ \cE_\varepsilon \triangleq \inf_{\gamma \in \Gamma} \max_{t \in [-1,1]} E_\varepsilon[\gamma(t)] \]
	satisfy
	\[ 0 < \liminf_{\varepsilon \downarrow 0} \cE_{\varepsilon} \leq \limsup_{\varepsilon \downarrow 0} \cE_{\varepsilon} < \infty \]
	and that there exist $u_\varepsilon \in W^{1,2}(\Sigma)$ such that $E_\varepsilon[u_\varepsilon] = \cE_\varepsilon$ and $\delta E_\varepsilon[u_\varepsilon] = 0$ for all $\varepsilon > 0$; $\ind(u_\varepsilon) \leq 1$, $u_\varepsilon \in C^\infty(\Sigma)$, and $|u_\varepsilon| < 1$ are all standard. 

	{\bf Step 2: Partial regularity of the limit.} It follows from Hypothesis \eqref{assu:double.well.potential.convex} and the non-triviality of $u_\varepsilon$ that $|u_\varepsilon| < 1$ for all $\varepsilon$. All hypotheses needed to employ Theorem \ref{theo:main.known} are satisfied, so indeed we do obtain a limiting stationary integral 1-varifold $\energyunit^{-1} V^\infty$. Partial regularity follows from Lemma \ref{lemm:index.1.crude.picture}.
	
	{\bf Step 3: Local convergence near $p_*$.} Without loss of generality, we may suppose that we're working on a sequence $\{(u_i, \varepsilon_i)\}_{i=1,2,\ldots}$ with $\ind(u_i) = 1$ and with $\sing \support \Vert V^\infty \Vert = \{ p_* \}$. (For, if $\sing \support \Vert V^\infty \Vert = \emptyset$, there is no singular point. Likewise, if $\ind(u_i) = 0$ along a subsequence then $\sing \support \Vert V^\infty \Vert = \emptyset$ by Theorem \ref{theo:index.0.main}.) We are now precisely in the setting of Proposition \ref{prop:index.1.main}, and the result follows by combining the proposition with Lemma \ref{lemm:app.gmt.density.2.point.1.varifold} of the appendix.
\end{proof}

\appendix

\section{Geometric measure theory} \label{sec:app.gmt}

In this section we briefly recall some basic facts about geometric measure theory. (We refer the reader to \cite{Simon83} for a thorough treatment.)

\begin{defi}[$k$-varifolds, {\cite[Chapter 8, \S 38]{Simon83}}] \label{defi:app.gmt.varifold}
	Let $(\Sigma^n, g)$ be a complete Riemannian manifold and $U \subset \Sigma \setminus \partial \Sigma$ be open. We say $V$ is a $k$-varifold on $U$ if it is a Radon measure on $\grassmanian_k(U)$. We denote by $\Vert V \Vert$ the Radon measure induced by $V$ on $U$ under the projection $\pi : \grassmanian_k(U) \to U$, i.e.,
	\[ \Vert V \Vert(B) \triangleq V(\pi^{-1} B). \]	
\end{defi}

\begin{defi}[First variation, {\cite[Chapter 8, \S 39]{Simon83}}] \label{defi:app.gmt.varifold.first.variation}
	Let $(\Sigma^n, g)$ be a complete Riemannian manifold, $U \subset \Sigma \setminus \partial \Sigma$ be open, and $V$ be a $k$-varifold in $U$. The first variation of a $V$ in $U$ is
	\[ \delta V(\mathbf{X}) \triangleq \int_{\grassmanian_k(U)} \divg_T \mathbf{X}(x) \, dV(x, T), \; \mathbf{X} \in C^1_c(U; T\Sigma). \]
	We call $V$ stationary in $U$ if $\delta V(\mathbf{X}) = 0$ for all $X \in C^1_c(U; T\Sigma)$.
\end{defi}

\begin{defi}[Varifold density] \label{defi:app.gmt.varifold.density}
	Let $(\Sigma^n, g)$ be a complete Riemannian manifold, $U \subset \Sigma \setminus \partial \Sigma$ be open, and $p \in U$. For a $k$-varifold $V$ in $U$, we define the $k$-density of $V$ at $p \in U$ on scale $r$ to be
	\[ \Theta^k(\Vert V \Vert, p, r) \triangleq \frac{\Vert V \Vert(B_r(p) \cap U)}{\omega_k r^k}, \]
	provided $\dist_g(p, \partial U) < r$; here $\omega_k$ denotes the $k$-dimensional Lebesgue measure of $B_1(0) \subset \RR^k$. Likewise, we define the density at $p \in U$ to be
	\[ \Theta^k(\Vert V \Vert, p) \triangleq \lim_{r \downarrow 0} \Theta^k(\Vert V \Vert, p, r), \]
	provided the limit exists.
\end{defi}

\begin{lemm}[{\cite[Lemma 40.5]{Simon83}}] \label{lemm:app.gmt.varifold.density}
	Let $(\Sigma^n, g)$ be a complete Riemannian manifold, $U \subset \Sigma \setminus \partial \Sigma$ be open, and $V$ be a stationary $k$-varifold in $U$, then $\Theta^k(\Vert V \Vert, \cdot)$ exists everywhere on $U$.
\end{lemm}

Recalling the definition of a countably $k$-rectifiable set from \cite[Chapter 3]{Simon83}, we also proceed to define:

\begin{defi}[Integral varifold, {\cite[Chapter 4, \S 15]{Simon83}}] \label{defi:app.gmt.integral.varifold}
	Let $(\Sigma^n, g)$ be a complete Riemannian manifold and $U \subset \Sigma \setminus \partial \Sigma$ be open. A $k$-varifold $V$ in $U$ is integral if, for a countably $k$-rectifiable $S \subset U$,
	\[ V(f) = \int_S f(x, T_x S) \, \theta(x) \, d\cH^k(x) \text{ for all } f \in C^0_c(\grassmanian_k(U)), \]
	with $\theta \in \{ 0, 1, \ldots \}$ $\cH^k$-a.e. on $S$, and $\theta \in L^1_{\loc}(\cH^k \restr S)$.
\end{defi}

\begin{defi}[Regular, singular sets]
	Let $(\Sigma^n, g)$ be a complete Riemannian manifold, and $T \subset \Sigma$ be countably $k$-rectifiable. We denote
	\begin{multline*}
		\reg \support \Vert V \Vert = \Big\{ p \in \support \Vert V \Vert : \text{there exists } r > 0 \text{ such that} \\
		B_r(p) \cap \support \Vert V \Vert = \text{a smooth} \\
		\text{embedded } k\text{-dimensional submanifold} \Big\}, 
	\end{multline*}
	and we denote its complement within $\support \Vert V \Vert$ as $\sing \support \Vert V \Vert$.
\end{defi}

The following lemma is a simple fact in geometric measure theory; its proof is simple but not readily available in the literature, so we include it here for the reader's convenience.

\begin{lemm} \label{lemm:app.gmt.density.2.point.1.varifold}
	Let $(\Sigma^2, g)$ be a complete Riemannian manifold, $V$ be a stationary integral 1-varifold, $U \subset \Sigma \setminus \partial \Sigma$ be open, $\support \Vert V \Vert \cap U$ singular, $\support \Vert V \Vert \cap U \setminus \{p \}$ smooth for some $p \in U$ with $\Theta^1(\Vert V \Vert, p) = 2$. Then $\support \Vert V \Vert \cap U$ is the union of two smooth embedded geodesics $\Gamma_1$, $\Gamma_2$, with $\Gamma_1 \cap \Gamma_2 = \{ p \}$ and $\partial \Gamma_1 \cup \partial \Gamma_2 \subset \partial U$.
\end{lemm}
\begin{proof}
	From \cite{AllardAlmgren76} we know that
	\[ V \restr \grassmanian_1(U) = \sum_{i=1}^4 \duline{v}(\ell_i, \cH^1 \restr \ell_i), \]
	where $\{ \ell_i \}_{i=1,\ldots,4}$ are (not necessarily distinct) geodesic rays with endpoints $p$ and $q_i \in \partial U$. Denote $\{ \mathbf{v}_i \}_{i=1,\ldots,4} \subset \{ \mathbf{v} \in T_p \Sigma : \Vert \mathbf{x} \Vert = 1 \}$ the corresponding initial velocity vectors of $\{ \ell_i \}_{i=1,\ldots,4}$. Since $\support \Vert V \Vert \setminus \{p\}$ is smooth and $\support \Vert V \Vert$ is singular, the vectors $\{ \mathbf{v}_i \}_{i=1,\ldots,4}$ are all distinct and (after possibly relabeling them)
	\[ \mathbf{v}_1 + \mathbf{v}_2 = \mathbf{v}_3 + \mathbf{v}_4 = \mathbf{0}. \]
	By elementary considerations in Riemannian geometry, the pairs of geodesic rays $(\ell_1, \ell_2)$ and $(\ell_3, \ell_4)$ join up smoothly at $p$ to yield $\Gamma_1$, $\Gamma_2$ with the properties postulated in the statement.
\end{proof}

\section{Morse index under quadratic area growth} \label{sec:app.pde.morse.index.spaces.quadratic.area.growth}

In this section we will study general Schr\"odinger operators
\begin{equation} \label{eq:app.pde.schroedinger}
	L \triangleq -\Delta_g + V, \text{ where } V \in C^\infty_{\loc}(\Sigma) \cap L^\infty(\Sigma)
\end{equation}
on complete, noncompact Riemannian manifolds without boundary, and with quadratic volume growth; the latter condition means that their volume measure $\upsilon_g$ satisfies
\begin{equation} \label{eq:app.pde.quadratic.area.growth}
	\upsilon_g(B_R(p)) \leq c R^2,\; \text{for all} \; p \in \Sigma^n, R \geq 1.
\end{equation}

We associate with $L$ the quadratic form $\cQ : W^{1,2}(\Sigma) \otimes W^{1,2}(\Sigma) \to \RR$,
\begin{equation} \label{eq:app.pde.bilinear.form}
	\cQ(\zeta, \psi) \triangleq \int_\Sigma \left[ \langle \nabla \zeta, \nabla \psi \rangle + V \zeta \psi \right] \, d\upsilon_g \text{, } \zeta, \psi \in W^{1,2}(\Sigma).
\end{equation}
The corresponding Rayleigh quotient is $\cQ : W^{1,2}(\Sigma) \setminus \{ 0 \} \to \RR$,
\begin{equation} \label{eq:app.pde.rayleigh.quotient}
	\cR[\zeta] \triangleq \frac{\cQ(\zeta, \zeta)}{\Vert \zeta \Vert_{L^2(\Sigma)}^2} \text{, } \zeta \in W^{1,2}(\Sigma) \setminus \{ 0 \}.
\end{equation}

\begin{defi}[Morse index, nullity]
	Let $(\Sigma^n, g)$ be complete, noncompact, without boundary, with quadratic volume growth \eqref{eq:app.pde.quadratic.area.growth}, let $L \triangleq -\Delta_g + V$, with $V \in C^\infty_{\loc}(\Sigma) \cap L^\infty(\Sigma)$, and suppose $\Omega \subseteq \Sigma$ is an open, connected, Lipschitz domain. We define the Morse index of $L$ on $\Omega$ as
	\begin{multline} \label{eq:app.pde.index}
		\ind(L; \Omega) \triangleq \sup \Big\{ \dim V : V \subset W^{1,2}_0(\Omega) \text{ a subspace such that} \\
			\cQ(\zeta, \zeta) < 0 \text{ for all } \zeta \in W^{1,2}_0(\Omega) \setminus \{ 0 \} \Big\}
	\end{multline}
	and the nullity of $L$ on $\Omega$ as
	\begin{equation} \label{eq:app.pde.nullity}
		\nul(L; \Omega) \triangleq \dim \{ u \in W^{1,2}_0(\Omega) : Lu = 0 \text{ weakly in } \Omega \}.
	\end{equation}
\end{defi}

The Morse index counts the dimensionality of the space instabilities for a particular critical point. Heuristically, this corresponds to the number of negative eigenvalues, counted with multiplicity.

Some classical results on the Morse index of Schr\"odinger operators on compact domains generalize to the noncompact setting, provided we work under the quadratic area growth assumption \eqref{eq:app.pde.quadratic.area.growth}. We quote below, without proof, two results that are needed in the paper. A rigorous proof of both results can be found in the author's Ph.D. thesis \cite{Mantoulidis17}.

\begin{theo}[Noncompact Courant nodal domain theorem, {\cite[Theorem 4.3.7]{Mantoulidis17}}] \label{theo:app.pde.courant.nodal.domain}
	Let $(\Sigma^n, g)$ be complete, noncompact, without boundary, and with quadratic volume growth \eqref{eq:app.pde.quadratic.area.growth}. Suppose the open, connected, Lipschitz domain $\Omega$ can be partitioned into open, connected, disjoint, Lipschitz domains $\Omega_1, \ldots, \Omega_m$. For any $L$ as in \eqref{eq:app.pde.schroedinger},
	\begin{equation} \label{eq:app.pde.courant.nodal.domain.theorem} 
		\ind(L; \Omega) \geq \sum_{i=1}^{m-1} \left( \ind(L; \Omega_i) + \nul(L; \Omega_i) \right) + \ind(L; \Omega_m).
	\end{equation}
\end{theo}

The following proposition is motivated by Ghoussoub-Gui's original proof of Conjecture \ref{conj:degiorgi.monotone.conjecture} in $\RR^2$ \cite[Theorem 1.1]{GhoussoubGui98}.

\begin{lemm}[{\cite[Lemma 4.3.8]{Mantoulidis17}}]  \label{lemm:app.pde.unstable.past.nodal.domain}
	Let $(\Sigma^n, g)$ be complete, noncompact, without boundary, with quadratic area growth \eqref{eq:app.pde.quadratic.area.growth}, and let $L$ be as in \eqref{eq:app.pde.schroedinger}. Suppose $u \not \equiv 0$ is a bounded Jacobi field. If $\Omega \subseteq \Sigma$ is an open, connected, Lipschitz domain, with $u|_{\partial \Omega} \equiv 0$, then on every open, connected, Lipschitz $\Omega' \supsetneq \Omega$, $\ind(L; \Omega') \geq 1$.
\end{lemm}

\section{Wang-Wei curvature estimates on manifolds} \label{sec:app.wang.wei.manifolds}

The recent novel work of Wang-Wei \cite{WangWei17} is performed in the context of solutions to \eqref{eq:pde} on two-dimensional Euclidean space. In this appendix we outline necessary modifications that will allow \cite[Theorem 3.7]{WangWei17} to go through as Theorem \ref{theo:wang.wei.epsilon.curvature} in our setting. Seeing as to how computations in \cite{WangWei17} were carried out in Fermi coordinates, this generalization is, for the most part, straightforward---provided one sets everything up correctly, as we aim to do here. We describe this in some detail here, starting with introducing notation that will uniformize ideas with the current paper.

First, we may assume that our double-well potential is rescaled so that
\[ W''(\pm 1) = 2, \]
which allows to match the asymptotic analysis from \cite{WangWei17} verbatim. We may further assume that $U = B_2(\mathbf{0}) \subset \RR^2$, whose coordinates are $(x^1, x^2)$, and that the metric $g$ on $U$ is $C^\infty$ close to the flat metric on $B_2(\mathbf{0}) \subset \RR^2$. 

We work in the rescaled setting $(\widetilde{U}, \widetilde{g})$, with the rescaled function
\[ \widetilde{u}(\mathbf{x}) \triangleq u(\varepsilon \mathbf{x}), \; \mathbf{x} \in \widetilde{U}, \]
where $\widetilde{U} \triangleq B_{3\varepsilon^{-1}}(\mathbf{0})$  and the rescaled metric is $\widetilde{g} \triangleq \varepsilon^{-2} g$. If $\widetilde{\Gamma}_\alpha$ is a component of $\{ \widetilde{u} = 0 \}$ (denoted $\Gamma_\alpha$ and $\{ u = 0 \}$, respectively, in \cite{WangWei17}) that is graphical in the $(x^1,x^2)$-coordinates over
\[ [-2\varepsilon^{-1}, 2\varepsilon^{-1}] \times \{0\} \subset \widetilde{U} \]
using a graphing function
\begin{equation} \label{eq:app.wang.wei.manifolds.i}
	\widetilde{f}_\alpha : [-2\varepsilon^{-1}, 2\varepsilon^{-1}] \to \RR, \; |\widetilde{f}_\alpha'| + \varepsilon^{-1} |\widetilde{f}_\alpha''| \leq C_1, 
\end{equation}
$C_1 = C_1(g, C)$, where $C$ is the uniform curvature estimate we made on the level sets, then we can construct Fermi coordinates $(\widetilde{\Pi}_\alpha, \widetilde{d}_\alpha)$ as
\[ (\widetilde{\Pi}_\alpha, \widetilde{d}_\alpha) \mapsto \exp_{(\widetilde{\Pi}_\alpha,\widetilde{f}_\alpha(\widetilde{\Pi}_\alpha))} ( \widetilde{d}_\alpha \nu^\alpha ), \; |\widetilde{\Pi}_\alpha| \leq 2 \varepsilon^{-1}, \]
where $\nu^\alpha$ denotes the upward pointing unit normal to $\widetilde{\Gamma}_\alpha$ with respect to $\widetilde{g}$.   (These correspond to $\Pi_\alpha$ and $d_\alpha$ in \cite[Lemma 8.3]{WangWei17}.) Denote, in Fermi coordinates,
\[ \widetilde{\Gamma}_{\alpha,z} \triangleq \{ \widetilde{d}_\alpha = z \}. \]
For simplicity, we will write
\[ (x, z) \text{ in place of } (\widetilde{\Pi}_\alpha, \widetilde{d}_\alpha) \]
whenever we center on a fixed $\widetilde{\Gamma}_\alpha$. The metric, in these coordinates, is
\[ \widetilde{g} = \begin{bmatrix} 
	\widetilde{g}_{xx} & 0 \\
	0 & 1 \end{bmatrix}. \]
We note that the sectional curvatures satisfy
\begin{equation} \label{eq:app.wang.wei.manifolds.ii}
	|\sect_{\widetilde{g}}| + \varepsilon^{-1} |\partial \sect_{\widetilde{g}}| \leq C_2 \varepsilon^2,
\end{equation}
where $C_2 = C_2(g)$. Let $\widetilde{\sff}_{\alpha,z}$ denote the second fundamental form $\widetilde{\Gamma}_{\alpha,z}$, and $\widetilde{H}_{\alpha,z}$ the mean curvature scalar. We have $\widetilde{H}_{\alpha,0} = O(\varepsilon)$. Combined with the Riccati equation
\[ \frac{\partial}{\partial z} \widetilde{H}_{\alpha,z} = - \widetilde{H}_{\alpha,z}^2 - \sect_{\widetilde{g}}, \]
a straightforward ODE comparison, and  \eqref{eq:app.wang.wei.manifolds.ii}, we get
\begin{equation} \label{eq:app.wang.wei.manifolds.iii}
	|\widetilde{H}_{\alpha,z}| \leq C_3 \varepsilon \text{ provided } |\widetilde{\Pi}_\alpha| \leq \varepsilon^{-1}, \; |z| \leq \delta \varepsilon^{-1},
\end{equation}
where $\delta = \delta(g, C_1, C_2) \ll 1$, $C_3 = C_3(g, C_1, C_2)$.

We assume $|z| \leq \delta \varepsilon^{-1}$ in all that follows in this appendix and that we're working within $|\widetilde{\Pi}_\alpha| \leq \varepsilon^{-1}$.

\begin{center}
	{\bf Adjustments to \cite[Section 8]{WangWei17}}
\end{center}

\cite[Lemma 8.1, (8.5)-(8.7)]{WangWei17} go through without change, since \eqref{eq:app.wang.wei.manifolds.iii} above gives the required zero-th order curvature bound.

None of the pointwise identities \cite[(8.2)-(8.4)]{WangWei17} continue to hold true, seeing as to how we're no longer in the flat setting, so we need a Riemannian geometric approach to rederive \cite[(8.8)-(8.11)]{WangWei17}. Note, first, that \cite[(8.9)-(8.10)]{WangWei17} are an immediate consequence of \eqref{eq:app.wang.wei.manifolds.iii} combined with the fact that
\[ \left( \cL_{\partial/\partial z} \widetilde{g} \right)|_{T^* \widetilde{\Gamma}_{\alpha,z} \otimes T^* \widetilde{\Gamma}_{\alpha,z}} = 2 \widetilde{\sff}_{\alpha,z}. \]
From the Riccati equation, \eqref{eq:app.wang.wei.manifolds.ii}, and \eqref{eq:app.wang.wei.manifolds.iii}, we find that
\begin{equation} \label{eq:app.wang.wei.manifolds.iv}
	\frac{\partial}{\partial z} \widetilde{H}_{\alpha,z} = O(\varepsilon^2).
\end{equation}
This readily implies \cite[(8.8), (8.11)]{WangWei17}. By differentiating the Riccati identity in $x$, we get the evolution equation
\[ \frac{\partial}{\partial z} \left( \frac{\partial}{\partial x} \widetilde{H}_{\alpha,z} \right) = - 2 \widetilde{H}_{\alpha,z} \left( \frac{\partial}{\partial x} \widetilde{H}_{\alpha,z} \right) - \frac{\partial}{\partial x} \sect_{\widetilde{g}}. \]
We already know that $\frac{\partial}{\partial x} \widetilde{H}_{\alpha,0} = O(\varepsilon)$ from \cite[(8.5)]{WangWei17}. Combined with the evolution equation above and \eqref{eq:app.wang.wei.manifolds.ii}, we conclude
\begin{equation} \label{eq:app.wang.wei.manifolds.v}
	\frac{\partial}{\partial x} \widetilde{H}_{\alpha,z} = O(\varepsilon).
\end{equation}
From \eqref{eq:app.wang.wei.manifolds.i} we find that
\begin{equation} \label{eq:app.wang.wei.manifolds.vi}
	\frac{\partial}{\partial x} \widetilde{g}_{xx} = O(\varepsilon) \text{ on } \widetilde{\Gamma}_{\alpha}.
\end{equation}
Combined with \eqref{eq:app.wang.wei.manifolds.v} and the evolution of $\widetilde{g}_{xx}$ with respect to $\frac{\partial}{\partial z}$, we also get
\begin{equation} \label{eq:app.wang.wei.manifolds.vii}
	\frac{\partial}{\partial x} \widetilde{g}_{xx}(x, z) - \frac{\partial}{\partial x} \widetilde{g}_{xx}(x, 0) = O(\varepsilon |z|).
\end{equation}
Given \eqref{eq:app.wang.wei.manifolds.vii}, \cite[Lemma 8.2, (8.13)]{WangWei17} goes through unchanged.

\cite[Lemma 8.3]{WangWei17} goes through as well, subject to  modifications in step 3. Namely, \cite[(8.15)-(8.17)]{WangWei17} are shown as in the paper, but \cite[(8.14), (8.18)]{WangWei17} require a slightly more geometric approach. By \eqref{eq:app.wang.wei.manifolds.iii}, 
\begin{equation} \label{eq:app.wang.wei.manifolds.viii}
	|\nabla^{\widetilde{g}}_{\nabla^{\widetilde{g}} \widetilde{d}_\alpha} \langle \nabla^{\widetilde{g}} \widetilde{d}_\beta, \nabla^{\widetilde{g}} \widetilde{d}_\alpha \rangle_{\widetilde{g}}| = |\nabla^2_{\widetilde{g}} \widetilde{d}_\beta(\nabla^{\widetilde{g}} \widetilde{d}_\alpha, \nabla^{\widetilde{g}} \widetilde{d}_\alpha)| = O(\varepsilon).
\end{equation}
Then, arguing as in step 1 of the lemma, it is simple to check that
\[ \langle \nabla^{\widetilde{g}} \widetilde{d}_\beta, \nabla^{\widetilde{g}} \widetilde{d}_\alpha \rangle_{\widetilde{g}} \geq 1 - O(\varepsilon^{\frac{1}{2}} \log^{\frac{1}{2}} \varepsilon^{-1}) \text{ on } \widetilde{\Gamma}_\alpha. \]
Consider a geodesic starting from a point $q \in \widetilde{\Gamma}_\alpha$ and ending at its closest point $q' \in \widetilde{\Gamma}_\beta$. Vary this geodesic so that $q$ gets pushed in the $\nabla^{\widetilde{g}} \widetilde{d}_\alpha$ direction. Denote normal to this geodesic by $\mathbf{N}$. The Jacobi field $\mathbf{V}$ of this geodesic variation then satisfies
\begin{equation} \label{eq:app.wang.wei.manifolds.ix}
	|\langle \mathbf{V}_q, \mathbf{N}_q \rangle_{\widetilde{g}}| = O(\varepsilon^{\frac{1}{2}} \log^{\frac{1}{2}} \varepsilon^{-1}).
\end{equation}
Moreover, if $\mathbf{V}'_q$ denotes differentiation along this geodesic at the point $q$, then \eqref{eq:app.wang.wei.manifolds.viii} implies $|\langle \mathbf{V}'_q, \mathbf{N}_q \rangle_{\widetilde{g}}| = O(\varepsilon)$. From the Jacobi equation for $\mathbf{V}$, \eqref{eq:app.wang.wei.manifolds.iii}, and the fundamental theorem of calculus, 
\[ |\langle \mathbf{V}_{q'}, \mathbf{N}_{q'} \rangle_{\widetilde{g}}| = O(\varepsilon^{\frac{1}{2}} \log^{\frac{1}{2}} \varepsilon^{-1}), \]
at the other endpoint $q'$, recovering \cite[(8.23)]{WangWei17}. Integrating gives \cite[(8.14), (8.24)]{WangWei17}. Finally, \cite[(8.18)]{WangWei17} follows from \eqref{eq:app.wang.wei.manifolds.viii}-  \eqref{eq:app.wang.wei.manifolds.ix}. This concludes \cite[Lemma 8.3]{WangWei17}.

\begin{center}
	{\bf Adjustments to \cite[Section 14]{WangWei17}}
\end{center}

The left hand side of \cite[(14.1)]{WangWei17} is understood to be $\widetilde{H}_{\alpha,0}$.

\begin{center}
	{\bf Adjustments to \cite[Section 17]{WangWei17}}
\end{center}

The main result of \cite[Section 17]{WangWei17}, \cite[Proposition 17.1]{Wang14}, goes through with its statement unchanged, but with modifications to its proof; the issue at hand is that coordinate derivatives of $u$ are no longer Jacobi fields in the curved setting. Note that, unlike in the original paper, we will continue to work in the stretched coordinate system here, i.e., with $\widetilde{u}$ on $(\widetilde{U}, \widetilde{g})$ instead of $u$ on $(U, g)$, seeing as to how all our modifications have been stated relative to the prior.

For this section we introduce the following modifications. First, we define a truncated, almost-Jacobi field
\[ \overline{\varphi} \triangleq \indt{\widetilde{D}_{\alpha}} \frac{\partial \widetilde{u}}{\partial x^2}, \]
where $\widetilde{D}_\alpha$ is the component of $\{ \frac{\partial \widetilde{u}}{\partial x^2} > 0 \}$ containing $\widetilde{\Gamma}_{\alpha}$.

Equation \cite[(17.1)]{WangWei17} is no longer true. Instead, differentiating \eqref{eq:pde} (rescaled to $\varepsilon = 1$) in the $x^2$ coordinate, where $(x^1, x^2)$ is the Euclidean coordinate chart for $\widetilde{U}$ relative to which \eqref{eq:app.wang.wei.manifolds.i} holds,
\begin{multline} \label{eq:app.wang.wei.manifolds.x}
	\implies \frac{\partial}{\partial x^2} \left( \frac{1}{\sqrt{g}} \frac{\partial}{\partial x^i} \left( \sqrt{g} g^{ij} \frac{\partial \widetilde{u}}{\partial x^j} \right) \right) = W''(\widetilde{u}) \frac{\partial \widetilde{u}}{\partial x^2} \\
	\iff \Delta_{\widetilde{g}} \left( \frac{\partial \widetilde{u}}{\partial x^2} \right) + O(\varepsilon)(|\partial^2 \widetilde{u}| + |\partial \widetilde{u}|) = W''(u) \frac{\partial \widetilde{u}}{\partial x^2};
\end{multline}
this is our replacement for \cite[(17.1)]{WangWei17}.

\cite[Lemma 17.2]{WangWei17}, \cite[Lemma 17.3]{WangWei17} go through (for stretched coordinates) with straightforward modifications; the pertinent chain of inequalities, in the stretched setting, is 
\begin{multline} \label{eq:app.wang.wei.manifolds.xi}
	\int_{B_{\varepsilon^{-1}/100}(x_\varepsilon)} \varepsilon |\nabla^{\widetilde{g}} (\overline{\varphi} \eta)|^2 + \varepsilon^{-1} W''(\widetilde{u}) \overline{\varphi}^2 \eta^2 \, d\upsilon_{\widetilde{g}} \\
	\leq \varepsilon \int_{B_{\varepsilon^{-1}/100}(x_{\varepsilon})} \overline{\varphi}^2 \, d\upsilon_{\widetilde{g}} \leq C.
\end{multline}
Next, let
\[ \widetilde{\Omega}_{\alpha} \triangleq \left\{ (x, z) : |x| < L, \; L < z < \rho_\varepsilon - L \right\}, \; \rho_\varepsilon \triangleq \widetilde{f}_{\alpha+1}(0) \leq \log \varepsilon^{-1}. \]
Assume that $L$ is large enough in order for
\[ \widetilde{\Omega}_\alpha \subset \{ |W''(\widetilde{u})| \geq \kappa \}, \]
with $\kappa$ is as in \eqref{assu:double.well.potential.convex}. By \cite[Lemma 17.1]{WangWei17}, $\rho_\varepsilon \gg 1$. Define $\widetilde{\varphi} : \widetilde{\Omega}_\alpha \to \RR$ as the unique solution to the system
\[ \begin{Bmatrix}
 	-\Delta_{\widetilde{g}} \widetilde{\varphi} + W''(\widetilde{u}) \widetilde{\varphi} = 0 \text{ in } \widetilde{\Omega}_\alpha \\
 	\widetilde{\varphi} = \overline{\varphi} \text{ on } \partial \widetilde{\Omega}_\alpha
\end{Bmatrix}. \]
Existence and uniqueness follow from $\ind(\widetilde{u}; \widetilde{\Omega}_\alpha) = 0$. Notice that
\begin{align*}
	& \int_{\widetilde{\Omega}_\alpha} \langle \nabla^{\widetilde{g}} \overline{\varphi}, \nabla^{\widetilde{g}} \widetilde{\varphi} \rangle_{\widetilde{g}} + W''(\widetilde{u}) \overline{\varphi} \widetilde{\varphi} \\
	& \qquad = \int_{\partial \widetilde{\Omega}_\alpha} \overline{\varphi} \langle \nabla^{\widetilde{g}} \widetilde{\varphi}, \nu_{\partial \widetilde{\Omega}_\alpha} \rangle_{\widetilde{g}} + \int_{\widetilde{\Omega}_\alpha} \overline{\varphi} \left( - \Delta_{\widetilde{g}} \widetilde{\varphi} + W''(\widetilde{u}) \widetilde{\varphi} \right) \\
	& \qquad = \int_{\partial \widetilde{\Omega}_\alpha} \widetilde{\varphi} \langle \nabla^{\widetilde{g}} \widetilde{\varphi}, \nu_{\partial \widetilde{\Omega}_\alpha} \rangle_{\widetilde{g}} \\
	& \qquad = \int_{\widetilde{\Omega}_\alpha}  \widetilde{\varphi} \Delta_{\widetilde{g}} \widetilde{\varphi} + |\nabla^{\widetilde{g}} \widetilde{\varphi}|^2 \\
	& \qquad = \int_{\widetilde{\Omega}_\alpha} |\nabla^{\widetilde{g}} \widetilde{\varphi}|^2 + W''(\widetilde{u}) \widetilde{\varphi}^2,
\end{align*}
so
\begin{align} \label{eq:app.wang.wei.manifolds.xii}
	& \int_{\widetilde{\Omega}_\alpha} |\nabla^{\widetilde{g}} \overline{\varphi}|^2 + W''(\widetilde{u}) \overline{\varphi}^2 - \int_{\widetilde{\Omega}_\alpha} |\nabla^{\widetilde{g}} \widetilde{\varphi}|^2 + W''(\widetilde{u}) \widetilde{\varphi}^2 \\
	& \qquad = \int_{\widetilde{\Omega}_\alpha} |\nabla^{\widetilde{g}} \overline{\varphi}|^2 + W''(\widetilde{u}) \overline{\varphi}^2 + \int_{\widetilde{\Omega}_\alpha} |\nabla^{\widetilde{g}} \widetilde{\varphi}|^2 + W''(\widetilde{u}) \widetilde{\varphi}^2 \nonumber \\
	& \qquad \qquad - 2 \int_{\widetilde{\Omega}_\alpha} \langle \nabla^{\widetilde{g}} \overline{\varphi}, \nabla^{\widetilde{g}} \widetilde{\varphi} \rangle_{\widetilde{g}} + W''(\widetilde{u}) \overline{\varphi} \widetilde{\varphi} \nonumber \\
	& \qquad = \int_{\widetilde{\Omega}_\alpha} |\nabla^{\widetilde{g}} (\overline{\varphi} - \widetilde{\varphi})|^2 + W''(\widetilde{u})(\overline{\varphi} - \widetilde{\varphi})^2 \nonumber \\
	& \qquad \geq \kappa \int_{\widetilde{\Omega}_\alpha} (\overline{\varphi} - \widetilde{\varphi})^2, \nonumber
\end{align}
where $\kappa$ is as in \eqref{assu:double.well.potential.convex}. Inequality \eqref{eq:app.wang.wei.manifolds.xii} is \cite[(17.5)]{WangWei17} in the curved setting.

Recalling
\[ \Delta_{\widetilde{g}} - \Delta = (g^{ij} - \delta^{ij}) \frac{\partial^2}{\partial x^i \partial x^j} + \frac{1}{\sqrt{g}} \frac{\partial}{\partial x^i} \left( \sqrt{g} g^{ij} \right) \frac{\partial}{\partial x^j} \]
and that
\[ \sup_{\widetilde{\Omega}_\alpha} \left[ \left| g^{ij}-\delta^{ij} \right| + \left| \frac{1}{\sqrt{g}} \frac{\partial}{\partial x^i}(\sqrt{g} g^{ij}) \right| \right] \to 0 \text{ as } \varepsilon \to 0, \]
it follows easily that the subsolution \cite[(17.6)]{WangWei17} goes through, with a perhaps slightly worse exponent which nevertheless converges to the exponent in \cite[(17.6)]{WangWei17} as $\varepsilon \to 0$.

\cite[Lemma 17.4]{WangWei17} goes through. We may thus finish the section by concluding that
\begin{align*}
	\int_{\widetilde{\Omega}_\alpha} (\overline{\varphi} - \widetilde{\varphi})^2 \, d\upsilon_{\widetilde{g}} 
	& \geq \int_{\substack{|x^1| < L/2 \\ \frac{(1+o(1)) \rho_\varepsilon}{2} < x^2 < \frac{3 \rho_\varepsilon}{4}}} \widetilde{\varphi}^2 \, d\upsilon_{\widetilde{g}} \\
	& \geq \frac{(2+o(1))\rho_\varepsilon L}{4} \exp \left( - \frac{1+o(1)}{2} \sqrt{2+o(1)} \rho_\varepsilon \right),
\end{align*}
so the result follows as in \cite[Section 17]{WangWei17} by combining the inequality above with \eqref{eq:app.wang.wei.manifolds.xi} and  \eqref{eq:app.wang.wei.manifolds.xii}.

\begin{center}
	{\bf Adjustments to \cite[Section 18]{WangWei17}}
\end{center}

The left hand side of \cite[(18.6)]{WangWei17} is understood to be $\widetilde{H}_{\alpha,0}$.

\begin{center}
	{\bf Adjustments to \cite[Section 19]{WangWei17}}
\end{center}

In our notation, $\lambda = \widetilde{g}_{xx}^{-1}$. Unlike in the flat setting, we do not have a precise pointwise expression for $\lambda$. Nonetheless,
\[ \frac{\partial \lambda}{\partial z} = - \frac{1}{\widetilde{g}_{xx}^2} \frac{\partial \widetilde{g}_{xx}}{\partial z} = - \frac{2 \widetilde{H}_{\alpha,z}}{\widetilde{g}_{xx}} = O(\varepsilon), \]
by \eqref{eq:app.wang.wei.manifolds.iii}. Likewise,
\[ \frac{\partial^2 \lambda}{\partial z^2} = - \frac{1}{\widetilde{g}_{xx}} \frac{\partial \widetilde{H}_{\alpha,z}}{\partial z} + \frac{2 \widetilde{H}_{\alpha,z}^2}{\widetilde{g}_{xx}} = O(\varepsilon^2), \]
by \eqref{eq:app.wang.wei.manifolds.iii}-\eqref{eq:app.wang.wei.manifolds.iv}. The remainder of the section goes through.

\begin{center}
	{\bf Adjustments to \cite[Section 20]{WangWei17}}
\end{center}

The reduction of \cite[Theorem 3.7]{WangWei17} to \cite[Proposition 20.1]{WangWei17} as outlined in \cite{WangWei17} goes through in the curved setting.

As in \cite{WangWei17}, we introduce the notation
\[ A_0(r) \triangleq \sup_{|x^1| < r} \exp \left( - \sqrt{2} \widetilde{D}_0(x^1) \right), \]
where $\widetilde{D}_0 = \max\{ \dist_{\widetilde{g}}(\cdot; \widetilde{\Gamma}_{-1}), \dist_{\widetilde{g}}(\cdot; \widetilde{\Gamma}_1) \}$, and
\[ R \triangleq \varepsilon^{-1}, \; r \in \left[ \frac{R}{2}, \frac{4R}{5} \right], \; \epsilon \triangleq A_0(r), \]
and we assume that
\begin{equation} \label{eq:app.wang.wei.manifolds.xiii}
	\epsilon \geq \varepsilon^{8/7}.
\end{equation}
In this section it will be important to be able to relate the geodesic curvature of curves that are graphical over an axis to the second derivatives of their graphing functions. While in the flat setting the relationship is straightforward,
\[ \widetilde{H}_\alpha = \frac{\widetilde{f}_\alpha''}{(1+|\widetilde{f}_\alpha'|^2)^{3/2}}, \]
the relationship in the curved setting is less explicit. We start by noting that the 1-form
\[ \omega \triangleq \underbrace{-\widetilde{f}_\alpha'(x^1)}_{= \; \omega_1} dx^1 + \underbrace{1}_{= \; \omega_2} dx^2 \]
annihilates $\widetilde{\Gamma}_\alpha$. Therefore $\frac{\omega}{|\omega|_{\widetilde{g}}}$ is a $\widetilde{g}$-unit-normal 1-form to $\widetilde{\Gamma}_\alpha$, so, by definition,
\[ \widetilde{H}_\alpha = \widetilde{g}^{xx} \left( \nabla^{\widetilde{g}} \frac{\omega}{|\omega|_{\widetilde{g}}} \right)\left( \frac{\partial}{\partial x^1}, \frac{\partial}{\partial x^1} \right). \]
Note that
\begin{align*}
	& \left( \nabla^{\widetilde{g}} \frac{\omega}{|\omega|_{\widetilde{g}}} \right)\left( \frac{\partial}{\partial x^1}, \frac{\partial}{\partial x^1} \right) \\
	& \qquad = \left( \frac{\nabla^{\widetilde{g}} \omega}{|\omega|_{\widetilde{g}}} - \nabla^{\widetilde{g}} |\omega|_g^2 \otimes \frac{\omega}{2 |\omega|_{\widetilde{g}}^3} \right)\left( \frac{\partial}{\partial x^1}, \frac{\partial}{\partial x^1} \right) \\
	& \qquad = \frac{\partial_{1} \omega_1 - \Gamma_{11}^1 \omega_1 - \Gamma_{11}^2 \omega_2}{|\omega|_{\widetilde{g}}} \\
	& \qquad \qquad - \partial_1 \left( \widetilde{g}^{11} \omega_1^2 + 2 \widetilde{g}^{12} \omega_1 \omega_2 + \widetilde{g}^{22} \omega_2^2 \right) \frac{\omega_1}{2 |\omega|^3_{\widetilde{g}}} \\
	& \qquad = \widetilde{f}_\alpha'' \left( - \frac{1}{|\omega|_{\widetilde{g}}} + \frac{\widetilde{g}^{11} (\widetilde{f}_\alpha')^2}{|\omega|^3_{\widetilde{g}}} - \frac{\widetilde{g}^{12} \widetilde{f}_\alpha'}{|\omega|^3_{\widetilde{g}}} \right) + O(|\partial \widetilde{g}|).
\end{align*}
We may arrange for:
\[ |\partial^2 \widetilde{g}| = O(\varepsilon^2), \; \widetilde{g}|_{(0,0)} = \delta, \; \partial \widetilde{g}|_{(0,0)} = 0, \text{ and } \widetilde{f}_0(0) = \widetilde{f}_0'(0) = 0. \]
Next, we estimate $|\partial \widetilde{g}|$ on $\widetilde{\Gamma}_0 \cap \{ |x^1| \leq K\epsilon^{-1/2} \}$; recalling  \eqref{eq:app.wang.wei.manifolds.xiii},
\[ |\partial \widetilde{g}| = O(\varepsilon^2 \epsilon^{-1/2}) = O(\epsilon^{10/8}) \text{ on } \widetilde{\Gamma}_0 \cap \{ |x^1| \leq K\epsilon^{-1/2} \}. \]
Therefore,
\begin{equation} \label{eq:app.wang.wei.manifolds.xiv}
	\widetilde{H}_0 = (-1+o(1)) \widetilde{f}_0'' + O(\epsilon^{10/8}) \text{ on } \widetilde{\Gamma}_0 \cap \{ |x^1| \leq K \epsilon^{-1/2} \}.
\end{equation}
This recovers \cite[(20.6)]{WangWei17}. To perform the same estimate on $\widetilde{\Gamma}_{\pm 1} \cap \{ |x^1| \leq K \epsilon^{-1/2} \}$, we first recall that, by \eqref{eq:app.wang.wei.manifolds.xiii},
\[ \sup_{|x^1| < r} \exp \left( - \sqrt{2} \widetilde{D}_0(x^1) \right) \geq \varepsilon^{8/7} \iff \inf_{|x^1|<r} \widetilde{D}_0(x^1) \leq \frac{8}{7\sqrt{2}} \log \varepsilon^{-1}. \]
Therefore, a crude estimate will tell us that
\[ \dist_{\widetilde{g}}(\cdot; \widetilde{\Gamma}_0) = O\left( \epsilon^{-1/2} \log \varepsilon^{-1} \right) \text{ on } \widetilde{\Gamma}_{\pm 1} \cap \{ |x^1| \leq K \epsilon^{-1/2} \}, \]
and thus
\[ |\partial \widetilde{g}| = O\left( \varepsilon^2 \epsilon^{-1/2} \log \varepsilon^{-1} \right) = O(\epsilon^{10/9}) \text{ on } \widetilde{\Gamma}_{\pm 1} \cap \{ |x^1| \leq K\epsilon^{-1/2} \}. \]
Thus, arguing as before,
\begin{equation} \label{eq:app.wang.wei.manifolds.xv}
	\widetilde{H}_1 = (-1+o(1)) \widetilde{f}_{\pm 1}'' + O(\epsilon^{10/9}) \text{ on } \widetilde{\Gamma}_{\pm 1} \cap \{ |x^1| \leq K\epsilon^{-1/2} \}.
\end{equation}
\cite[Lemma 20.3]{WangWei17} will then continue to be true in the curved setting, provided we make use of \eqref{eq:app.wang.wei.manifolds.xv}, instead of the explicit flat estimate on $\widetilde{f}_{\pm 1}''$ from \cite[(18.6)]{WangWei17}. Likewise, making use of \eqref{eq:app.wang.wei.manifolds.xiv}, both estimates \cite[(20.9)-(20.10)]{WangWei17} remain valid. Making use of \eqref{eq:app.wang.wei.manifolds.xiv}-\eqref{eq:app.wang.wei.manifolds.xv}, \cite[Lemma 20.4]{WangWei17} remains valid, too.

\cite[Lemma 20.5]{WangWei17} goes through verbatim. 

Finally, \cite[Lemma 20.6]{WangWei17} goes through provided we, again, replace the explicit flat estimate in \cite[(18.6)]{WangWei17} by \eqref{eq:app.wang.wei.manifolds.xv}.

\begin{center}
	{\bf Adjustments to \cite[Appendix F]{WangWei17}}
\end{center}

All steps, except step (12), go through verbatim. In step (12), one simply needs to estimate
\[ \left| \frac{\partial}{\partial z} \widetilde{g}_{xx}^{-1} \right|_{C^\theta} \]
in terms of
\[ \left| \frac{\partial}{\partial x} \frac{\partial}{\partial z} \widetilde{g}_{xx}^{-1} \right|_{C^0} \leq \left| \frac{\partial}{\partial x} \left( \frac{1}{\widetilde{g}_{xx}} \widetilde{H}_{\alpha,z} \right) \right| = O(\varepsilon), \]
where the last asymptotic follows from  \eqref{eq:app.wang.wei.manifolds.iii}, \eqref{eq:app.wang.wei.manifolds.v}.

\begin{center}
	{\bf Adjustments to \cite[Appendix G]{WangWei17}}
\end{center}

In step (1) we need to compute the commutator between the horizontal partial derivative and the Laplacian along $\widetilde{\Gamma}_{\alpha,z}$. Recalling
\[ \Delta_{\widetilde{g}} = \Delta_{\widetilde{\Gamma}_{\alpha,z}} + \widetilde{H}_{\alpha,z} \frac{\partial}{\partial z} + \frac{\partial^2}{\partial z^2}, \]
we see that
\[ \left[ \frac{\partial}{\partial x}, \Delta_{\widetilde{\Gamma}_{\alpha,z}} \right] = \left[ \frac{\partial}{\partial x}, \Delta_{\widetilde{g}} \right] - \left( \frac{\partial}{\partial x} \widetilde{H}_{\alpha,z} \right) \frac{\partial}{\partial z}, \]
which, by \eqref{eq:app.wang.wei.manifolds.ii}, \eqref{eq:app.wang.wei.manifolds.v}, re-confirms the estimate 
\[ \left[ \frac{\partial}{\partial x}, \Delta_{\widetilde{\Gamma}_{\alpha,z}} \right] \phi = O(\varepsilon)(|\partial^2 \phi| + |\partial \phi|) \]
needed for step (1).

Step (2) steps by simply replacing the use of \cite[(8.4)]{WangWei17}, which is false in curved space, with \eqref{eq:app.wang.wei.manifolds.v}.

Steps (3)-(7) require no adjustments.

Step (8) follows by using \eqref{eq:app.wang.wei.manifolds.iv} and  \eqref{eq:app.wang.wei.manifolds.v} instead of \cite[(8.4)]{WangWei17}.

Step (9) follows verbatim, and steps (10), (11) follow with the same modifications that have been already explained above. Finally, step (12) follows without any modifications.

%\bibliographystyle{alpha}
%\bibliography{ac-index-1}

\begin{thebibliography}{dRGGHP03}

\bibitem[AA76]{AllardAlmgren76}
W.~K. Allard and F.~J. Almgren, Jr.
\newblock The structure of stationary one dimensional varifolds with positive
  density.
\newblock {\em Invent. Math.} {\bf 34} (1976), no. 2, 83--97, MR0425741, Zbl 0339.49020.

\bibitem[AC00]{AmbrosioCabre00}
L.~Ambrosio and X.~Cabr\'e.
\newblock Entire solutions of semilinear elliptic equations in {$\bold R^3$}
  and a conjecture of {D}e {G}iorgi.
\newblock {\em J. Amer. Math. Soc.} {\bf 13} (2000), no. 4, 725--739, MR1775735, Zbl 0968.35041.

\bibitem[BJS79]{BersJohnSchechter79}
Lipman Bers, Fritz John, and Martin Schechter.
\newblock {\em Partial differential equations}.
\newblock American Mathematical Society, Providence, R.I., 1979.
\newblock With supplements by Lars G\.{a}rding and A. N. Milgram, With a
  preface by A. S. Householder, Reprint of the 1964 original, Lectures in
  Applied Mathematics, 3{{\rm{A}}}, MR0598466, Zbl 0514.35001.

\bibitem[CM16]{ChodoshMaximo16}
O.~Chodosh and D.~Maximo.
\newblock On the topology and index of minimal surfaces.
\newblock {\em J. Differential Geom.} {\bf 104} (2016), no. 3, 399--418, MR3568626, Zbl 1357.53016.

\bibitem[DG79]{DeGiorgi79}
E.~De~Giorgi.
\newblock Convergence problems for functionals and operators.
\newblock In {\em Proceedings of the {I}nternational {M}eeting on {R}ecent
  {M}ethods in {N}onlinear {A}nalysis ({R}ome, 1978)}, pages 131--188.
  Pitagora, Bologna, 1979, MR0533166, Zbl 0405.49001.

\bibitem[dPKP13]{DelPinoKowalczykPacard13}
M.~del Pino, M.~Kowalczyk, and F.~Pacard.
\newblock Moduli space theory for the {A}llen-{C}ahn equation in the plane.
\newblock {\em Trans. Amer. Math. Soc.} {\bf 365} (2013), no. 2, 721--766, MR2995371, Zbl 1286.35018.

\bibitem[dPKW11]{DelPinoKowalczykWei11}
M.~del Pino, M.~Kowalczyk, and J.~Wei.
\newblock On {D}e {G}iorgi's conjecture in dimension {$N\geq 9$}.
\newblock {\em Ann. of Math. (2)} {\bf 174} (2011), no. 3, 1485--1569, MR2846486, Zbl 1238.35019.

\bibitem[dPKW13]{DelPinoKowalczykWei13}
M.~del Pino, M.~Kowalczyk, and J.~Wei.
\newblock Entire solutions of the {A}llen-{C}ahn equation and complete embedded
  minimal surfaces of finite total curvature in {$\Bbb R^3$}.
\newblock {\em J. Differential Geom.} {\bf 93} (2013), no. 1, 67--131, MR3019512, Zbl 1275.53015.

\bibitem[dPKWY10]{DelPinoKowalczykWeiYang10}
M.~del Pino, M.~Kowalczyk, J.~Wei, and J.~Yang.
\newblock Interface foliation near minimal submanifolds in {R}iemannian
  manifolds with positive {R}icci curvature.
\newblock {\em Geom. Funct. Anal.} {\bf 20} (2010), no. 4, 918--957, MR2729281, Zbl 1213.35219.

\bibitem[dRGGHP03]{DelRioGuerraGarzaHumePadilla03}
H.~del Rio~Guerra, C.~E. Garza-Hume, and P.~Padilla.
\newblock Geodesics, soap bubbles and pattern formation in {R}iemannian
  surfaces.
\newblock {\em J. Geom. Anal.} {\bf 13} (2003), no. 4, 595--604, MR2005155, Zbl 1085.49047.

\bibitem[FMV13]{FarinaMariValdinoci13}
A.~Farina, L.~Mari, and E.~Valdinoci.
\newblock Splitting theorems, symmetry results and overdetermined problems for
  {R}iemannian manifolds.
\newblock {\em Comm. Partial Differential Equations} {\bf 38} (2013), no. 10, 1818--1862, MR3169764, Zbl 1287.58011.

\bibitem[GG98]{GhoussoubGui98}
N.~Ghoussoub and C.~Gui.
\newblock On a conjecture of {D}e {G}iorgi and some related problems.
\newblock {\em Math. Ann.} {\bf 311} (1998), no. 3, 481--491, MR1637919, Zbl 0918.35046.

\bibitem[GNY04]{GrigoryanNetrusovYau04}
A.~Grigor$\prime$yan, Y.~Netrusov, and S.-T. Yau.
\newblock Eigenvalues of elliptic operators and geometric applications.
\newblock In {\em Surveys in differential geometry. {V}ol. {IX}}, volume~9 of
  {\em Surv. Differ. Geom.}, pages 147--217. Int. Press, Somerville, MA, 2004, MR2195408, Zbl 1061.58027.

\bibitem[{Gua}18]{Guaraco15}
M.~A.~M. {Guaraco}.
\newblock {Min-max for phase transitions and the existence of embedded minimal
  hypersurfaces}.
\newblock {\em J. Differential Geom.} {\bf 108} (2018), no. 1, 91--133, MR3743704, Zbl 1387.49060.

\bibitem[Gui12]{Gui12}
C.~Gui.
\newblock Symmetry of some entire solutions to the {A}llen-{C}ahn equation in
  two dimensions.
\newblock {\em J. Differential Equations} {\bf 252} (2012), no. 11, 5853--5874, MR2911416, Zbl 1250.35078.

\bibitem[HT00]{HutchinsonTonegawa00}
J.~E. Hutchinson and Y.~Tonegawa.
\newblock Convergence of phase interfaces in the van der
  {W}aals-{C}ahn-{H}illiard theory.
\newblock {\em Calc. Var. Partial Differential Equations} {\bf 10} (2000), no. 1, 49--84, MR1803974, Zbl 1070.49026.

\bibitem[KLP12]{KowalczykLiuPacard12}
M.~Kowalczyk, Y.~Liu, and F.~Pacard.
\newblock The space of 4-ended solutions to the {A}llen-{C}ahn equation in the
  plane.
\newblock {\em Ann. Inst. H. Poincar\'e Anal. Non Lin\'eaire} {\bf 29} (2012), no. 5, 761--781, MR2971030, Zbl 1254.35219.

\bibitem[LS47]{LyusternikSchnirelmann47}
L.~Lyusternik and L.~Schnirelmman.
\newblock Topological methods in variational problems and their application to
  the differential geometry of surfaces.
\newblock {\em Uspehi Matem. Nauk (N.S.)} {\bf 2} (1947), no. 1(17), 166--217, MR0029532.

\bibitem[LW]{LiuWei18}
Y.~Liu and J.~Wei.
\newblock { A complete classification of finite Morse index solutions to
  elliptic sine-Gordon equation in the plane}, arXiv preprint (2018) 
\newblock arXiv:1806.06921.

\bibitem[LWW17]{LiuWangWei16}
Y.~{Liu}, K.~{Wang}, and J.~{Wei}.
\newblock {Global minimizers of the Allen-Cahn equation in dimension $n\geq
  8$}.
\newblock {\em J. Math. Pures Appl. (9)} {\bf 108} (2017), no. 6, 818--840, MR3723158, Zbl 1380.35071.

\bibitem[Man17]{Mantoulidis17}
C.~Mantoulidis.
\newblock {\em Geometric variational problems in mathematical physics}.
\newblock PhD thesis, Stanford University, June 2017.

\bibitem[PW13]{PacardWei13}
F.~Pacard and J.~Wei.
\newblock Stable solutions of the {A}llen-{C}ahn equation in dimension 8 and
  minimal cones.
\newblock {\em J. Funct. Anal.} {\bf 264} (2013), no. 5, 1131--1167, MR3010017, Zbl 1281.35046.

\bibitem[Sav09]{Savin09}
O.~Savin.
\newblock Regularity of flat level sets in phase transitions.
\newblock {\em Ann. of Math. (2)} {\bf 169} (2009), no. 1, 41--78, MR2480601, Zbl 1180.35499.

\bibitem[Sim83]{Simon83}
L.~Simon.
\newblock {\em Lectures on geometric measure theory}, volume~3 of {\em
  Proceedings of the Centre for Mathematical Analysis, Australian National
  University}.
\newblock Australian National University, Centre for Mathematical Analysis,
  Canberra, 1983, MR0756417, Zbl 0546.49019.

\bibitem[SS81]{SchoenSimon81}
R.~Schoen and L.~Simon.
\newblock Regularity of stable minimal hypersurfaces.
\newblock {\em Comm. Pure Appl. Math.} {\bf 34} (1981), no. 6, 741--797, MR0634285, Zbl 0497.49034.

\bibitem[Ton05]{Tonegawa05}
Y.~Tonegawa.
\newblock On stable critical points for a singular perturbation problem.
\newblock {\em Comm. Anal. Geom.} {\bf 13} (2005), no. 2, 439--459, MR2154826, Zbl 1105.35008.

\bibitem[TW12]{TonegawaWickramasekera12}
Y.~Tonegawa and N.~Wickramasekera.
\newblock Stable phase interfaces in the van der {W}aals--{C}ahn--{H}illiard
  theory.
\newblock {\em J. Reine Angew. Math.} {\bf 668} (2012), 191--210, MR2948876, Zbl 1244.49077.

\bibitem[{Wan}17a]{Wang14}
K.~{Wang}.
\newblock {A new proof of Savin's theorem on Allen-Cahn equations}.
\newblock {\em J. Eur. Math. Soc. (JEMS)} {\bf 19} (2017), no. 10, 2997--3051, MR3713000, Zbl 1388.35069.

\bibitem[{Wan}17b]{Wang15}
K.~{Wang}.
\newblock {Some remarks on the structure of finite Morse index solutions to the
  Allen-Cahn equation in $\mathbb{R}^2$}.
\newblock {\em NoDEA Nonlinear Differential Equations Appl.} {\bf 24} (2017), no. 5, Art. 58, 17
  pp., MR3690656, Zbl 1382.35109.

\bibitem[WW19]{WangWei17}
K.~{Wang} and J.~Wei.
\newblock {Finite Morse index implies finite ends}.
\newblock {\em Comm. Pure Appl. Math.} {\bf 72} (2019), no. 5, 1044--1119, MR3935478, Zbl 1418.35190.

\end{thebibliography}

\end{document}